\documentclass[preprint,12pt]{elsarticle}

\usepackage{epstopdf}
\usepackage{amssymb,amsmath}
\usepackage{amsthm}
\usepackage[pdftex,
           pdfstartview=FitH,
           colorlinks=true,
           citecolor=blue]{hyperref}

\newtheorem{theorem}{Theorem}[section]
\newtheorem{lemma}{Lemma}[section]

\newtheorem{proposition}{Proposition}[section]
\newtheorem{definition}{Definition}[section]
\newtheorem{remark}{Remark}[section]
\newtheorem{example}{Example}[section]
\journal{}

\begin{document}

\begin{frontmatter}



\cortext[cor1]{Corresponding author.}
\title{A second order dynamical system method for solving a maximal comonotone inclusion problem}

\author{Zengzhen Tan\fnref{scu}}
\ead{zengzhentan@foxmail.com}

\author{Rong Hu\fnref{cuit}}
\ead{ronghumath@aliyun.com}

\author{Yaping Fang\fnref{scu}\corref{cor1}}
\ead{ypfang@scu.edu.cn}
\address[scu]{Department of Mathematics, Sichuan University, Chengdu, Sichuan, P.R. China}
\address[cuit]{Department of Applied
Mathematics, Chengdu University of Information Technology, Chengdu, Sichuan, P.R. China}

\begin{abstract}
In this paper a second order dynamical system model is proposed for computing a zero of a maximal comonotone operator  in Hilbert spaces. Under mild conditions, we prove existence and uniqueness of a strong global solution of the proposed dynamical system. A proper tuning of the parameters can allow us to establish fast convergence properties of the trajectories generated by the dynamical system. The  weak convergence of the trajectory to a zero of the maximal comonotone operator is also proved. Furthermore, a discrete version of the dynamical system is considered and  convergence properties matching to  that of the dynamical system are established under a same framework. Finally, the validity of the proposed dynamical system and its discrete version is demonstrated by two numerical examples. 
\end{abstract}


\end{frontmatter}


\section{Introduction}

This paper focuses on solving the following inclusion problem
\begin{equation}\label{problem}
\mbox{find} \quad x^*\in \mathcal{H}\quad \mbox{such that} \quad 0\in \mathcal{A}(x^*),
\end{equation}
where $\mathcal{A}:\mathcal{H}\rightarrow 2^{\mathcal{H}}$ is a point-to-set operator and $\mathcal{H}$ be a real Hilbert space. We denote the solution set of the problem $(\ref{problem})$ by $\mbox{zer}  \mathcal{A}:= {\mathcal{A}}^{-1}(0)$ and assume it to be nonempty. The inclusion problem $(\ref{problem})$ though looks simple, it covers many important applications in scientific fields such as image processing, computer vision, machine learning, signal processing, optimization, equilibrium theory, economics, game theory, partial differential equations, statistics, and so on (see, e.g., \cite{Combettes,Bottou,Lan,Tseng,Solodov,Rockafellar,BauschkeC,Malitsky,Bhatia,Jabbar,Lin}). A general method for solving $(\ref{problem})$ is the proximal point algorithm first suggested by Martinet \cite{Martinet} in the 1970s for solving variational problems and further generalized by Rockafellar \cite{Rockafellar} to get today's version. In the proximal point algorithm, iterates $\{x_n\}$, $n\ge 1$, are generated by the following rule:
\begin{equation}\label{PPA}
x_{n+1}=J_{\gamma_{n}}^{A}(x_n),	 \quad \{\gamma_{n}\}\subset(0,+\infty),
\end{equation}
where $A:\mathcal{H}\rightarrow 2^{\mathcal{H}}$ is a maximally monotone operator, $J_{\gamma}^{A}$ is the resolvent of the operator $A$ with index $\gamma>0$ and $Id$ is the identity operator. 

When $A$ is a maximal monotone operator, Rockafellar \cite{Rockafellar} proves the sequence $\{x_n\}$ generated by $(\ref{PPA})$ is weakly convergent globally to a zero of $A$ when zer$A\neq \emptyset$, and the sequence of the regularization parameters $\gamma_n$ has a positive lower bound. Br$\acute{e}$zis and Lions \cite{Brezis} weaken the latter assumption to merely require the sequence of squares of the regularization parameters to be nonsummable. A related result due to Eckstein and Bertsekas \cite{DR}(partially based on Gol'shtein and Tret'yakov \cite{Golshtein}) is the following relaxed proximal point algorithm  
$$x_{n+1}=(1-\xi_n)x_n+\xi_n J_{\gamma_{n}}^{A}(x_n),$$
where $\{\xi_n\}_{n=0}^{\infty}\subseteq (0,2)$ is a sequence of over-or under-relaxation factors. Weak convergence is proved in \cite{DR} for an inexact version of the relaxed proximal point algorithm under a standard summable errors condition. The accelerated proximal point method proposed by Kim \cite{Kim}, based on the performance estimation problem approach of Drori-Teboulle \cite{Drori}, writes for initial iterates $\{x_0, z_0, z_{-1}\}\subset \mathcal{H}$ and for $n\geq 0$,
\begin{equation*}
\left\{
\begin{array}{rcl}
&&x_{n+1}=J_{\mu}^{A}(z_n),\\
&&z_{n+1}=x_{n+1}+\frac{n}{n+2}(x_{n+1}-x_n)+\frac{n}{n+2}(z_{n-1}-x_n).
\end{array}
\right.
\end{equation*}	 
This yields the following convergence rate $\|A_{\mu}(x_n)\|=O(n^{-1})$ (see \cite [Theorem 4.1] {Kim}). However, the above algorithm does not establish the convergence of the iterates.
Maing$\acute{e}$ \cite{Mainge} introduces a corrected relaxed inertial proximal algorithm with constant relaxation factors (CRIPA-S) which enters the following framework of sequences $\{z_n, x_n\}\subset \mathcal{H}$ generated from starting points $\{z_{-1},x_{-1},x_0\}$ by
\begin{equation}\label{CRIPAS}
\left\{
\begin{array}{rcl}
&&z_{n}=x_n+(1-\frac{a_1}{bn+\bar{c}})(x_n-x_{n-1})+(1-\frac{a_2}{bn+\bar{c}})(z_{n-1}-x_n),\\
&&x_{n+1}=\frac{1}{1+k_0}z_n+\frac{k_0}{1+k_0}J_{\lambda(1+k_0)}^{A}(z_n),
\end{array}
\right.
\end{equation}	
where constants $\{k_0,b,\bar{c},a_1,a_2\}\subset (0,+\infty)$ satisfy $a_2>2b$, $a_1>b+a_2$ and $\bar{c}>\mbox{max}\{a_1,a_2\}$. This yields the  weak convergence to some element of zer$A$ and the following convergence rate $\|x_{n+1}-x_n\|=o(n^{-1})$ and $\|A_{\lambda}(x_n)\|=o(n^{-1})$. In recent decades, many generalizations and modifications of the proximal point algorithm have been considered by authors; see for example \cite{DaoT,DaoP,KohlenbachU}.

However, all the above-mentioned works rely heavily on monotonicity of $A$. For the nonmonotone case, Combettes and Pennanen \cite{CP}, Iusem et al. \cite{Iusem} and Pennanen \cite{Pennanen} replace the monotonicity assumptions appearing in earlier work by a weaker hypomonotonicity (see Remark $\ref{hypo}$) condition to study local convergence of proximal point methods. Weaker forms of monotonicity have been considered first in Spingarn \cite{Spingarn} (1981), where conditions are given that guarantee the local convergence of the proximal point algorithm without requiring monotonicity of $A$. Recently, when $\mathcal{A}$ is a comonotone operator (see Definition $\ref{comonotone}$), Bartz et al. \cite{BartzD} prove the relaxed proximal point algorithm 
$$x_{n+1}\in Tx_n~~~\mbox{where} ~~T:=(1-\kappa)Id+\kappa J_\gamma \mathcal{A}$$
converges weakly to a point in zer$\mathcal{A}$ and the rate of asymptotic regularity of $T$ is $o(\frac{1}{\sqrt{n}})$. Kohlenbach \cite{Kohlenbach} defines the Halpern-type proximal point algorithm by
$$x_{n+1}:=\alpha_n u+(1-\alpha_n)J_{\gamma_n}^{\mathcal{A}}x_n\in C,$$
where $\{\alpha_n\}\subset (0,1]$, $\lim_{n\rightarrow +\infty}\alpha_n=0$, $\sum_{n=0}^{+\infty}\alpha_n=+\infty$, $u,x_0\in C$ and $C\subseteq \mathcal{H}$ is a nonempty closed and convex sunset. This yields $\{x_n\}$ strongly converges to the zero of $\mathcal{A}$.
For more results on comonotone operators, we refer the reader to \cite{Cai,Liu,BartzC,BartzP} and the references therein.

There is a long history of using dynamical systems to solve optimization problems \cite{AttouchP,Fiori,Helmke,Alvarez,AttouchM}. 
Su, Boyd and Cand$\grave{e}$s \cite{Su} introduce the second order in time evolution equation with vanishing damping is defined as for $t\geq t_0$ 
\begin{equation}\label{Su}
\ddot{x}(t)+\frac{\alpha}{t}\dot{x}(t)+\bigtriangledown f(x(t))=0,	
\end{equation}
where $\alpha>0$ and $f:\mathcal{H}\rightarrow \mathbb{R}$ is a convex $\mathcal{C}^{1}$ function. Provided $\alpha\geq 3$, they prove the fast convergence property $f(x(t))-\min_{\mathcal{H}}f=O(t^{-2})$.   This contribution is the starting point of significant research activities devoted to this type of dynamics, which is an improvement of Polyak's heavy ball method with friction
$$\ddot{x}(t)+\alpha\dot{x}(t)+\bigtriangledown f(x(t))=0.$$ 
Weak convergence of the trajectories generated by $(\ref{Su})$ when $\alpha>3$ has been shown by Attouch et al. \cite{Attouchz} and May \cite{May}, with the improved rate of convergence for the functional values $f(x(t))-\min_{\mathcal{H}}f=o(t^{-2})$ as $t\rightarrow +\infty$.

Later, in \cite{AttouchPR} Attouch et al. add a Hessian driven damping term in $(\ref{Su})$, which makes it naturally related to Newton's and Levenberg-Marquardt methods, and get the following dynamical system
\begin{equation}\label{Hes}
\ddot{x}(t)+\frac{\alpha}{t}\dot{x}(t)+\beta{\bigtriangledown}^2 f(x(t))\dot{x}(t)+\bigtriangledown f(x(t))=0,	
\end{equation}
where $\beta\geq 0$ and $f$ is a convex $\mathcal{C}^{2}$ function. For $\alpha\geq 3$ and $\beta>0$, fast convergence of the values $f(x(t))-\min_{\mathcal{H}}f=O(t^{-2})$ is obtained. The addition of the Hessian driven damping term not only retains the convergence properties of the Nesterov accelerated method, but also provides fast convergence of the gradients to zero. Several recent studies have been devoted to this subject, see \cite{Kim,AttouchJ,BotC,LinJ,Shib}.

Second order dynamics with vanishing damping have been considered also in the context of solving monotone inclusion problems (see \cite{AttouchP})
$$\ddot{x}(t)+\frac{\alpha}{t}\dot{x}(t)+A_{\lambda(t)}(x(t))=0,$$
where $A_{\lambda}$ is the Yosida regularization of $A$ of parameter $\lambda>0$. Attouch and Peypouquet\cite{AttouchP} prove that, under the condition $\lambda(t)\times \frac{{\alpha}^2}{t^2}>1$ for $t\geq t_0>0$, the trajectory of the dynamical system converges weakly to a zero of $A$ and $\|\dot{x}(t)\|=O(1/t)$, where $A$ is a maximally monotone operator. This evolution equation has been further developed in \cite{AttouchSC}, where, in analogy with the dynamics in $(\ref{Hes})$, an additional Newton-like correction term is considered
\begin{equation}\label{sys6}
\ddot{x}(t)+\frac{\alpha}{t}\dot{x}(t)+b\frac{d}{dt}(A_{\lambda(t)}(x(t)))+A_{\lambda(t)}(x(t))=0.	
\end{equation}

As far as we know, there are few papers using dynamical system methods to solve the nonmonotone inclusion problem. Therefore, the motivation of this paper is to establish the following second order dynamical system to solve the inclusion problem $(\ref{problem})$
\begin{equation}\label{DS}
\left\{
\begin{array}{rcl}
&&\mu(t)=\mathcal{A}_{\eta}x(t)\\
&&\ddot{x}(t)+\dot{\mu}(t)+\frac{\alpha}{t}\dot{x}(t)+\frac{\beta}{t}\mu(t)=0,
\end{array}
\right.
\end{equation}	
where $\mathcal{A}:\mathcal{H}\rightarrow 2^{\mathcal{H}}$ is a maximal comonotone  operator. It is worth noting that in \cite{AttouchSC} the operator $\mathcal{A}$ is maximally monotone whereas in our framework the operator $\mathcal{A}$ requires only maximal  comonotone, not necessarily monotone. Under the same conditions, the trajectories generated by the dynamical system $(\ref{DS})$ seems to have a better convergence behavior than the trajectories generated by $(\ref{sys6})$ as some numerical experiments show. In addition, the freedom of controlling the parameter $\beta$ in $(\ref{DS})$ is essential as the numerical experiments section show. Our main work is to study the asymptotic behavior of dynamical system $(\ref{DS})$ and obtain some fast convergence properties. By discretization of the dynamical system $(\ref{DS})$, a new numerical algorithm can be proposed for solving inclusion problem $(\ref{problem})$ where $\mathcal{A}$ is a maximal comonotone operator. In our setting, because of the singularity of the damping coefficient $\frac{\alpha}{t}$ at $t=0$, we always set the initial time $t_0>0$.

The remainder of this paper is organized as follows. Section 2 consists of some preliminary results. In Section 3, the existence and uniqueness of solutions of the considered dynamical system $(\ref{DS})$ are proved. Section 4 describes the weak convergence and relative convergence rate along the tarjectories of the system $(\ref{DS})$. A discrete version dynamical system and its weak convergence are presented in Section 5. Finally, some numerical experiments are reported in Section 6 to illustrate the obtained theoretical results.

\section{Preliminaries}

We will employ standard notations that generally follow \cite{BauschkeC}. Throughout, $\mathcal{H}$ is a real Hilbert space with the inner product $\langle \cdot, \cdot \rangle$ and its induced norm $\|\cdot\|$. The set of nonnegative integers is denoted by $\mathbb{N}$, the set of real numbers by $\mathbb{R}$, the set of nonnegative real numbers by $\mathbb{R}_{+}:=\{x\in \mathbb{R}|x\geq 0\}$, and the set of positive real numbers by $\mathbb{R}_{++}:=\{x\in \mathbb{R}|x > 0\}$. We use the notation $\mathcal{A}:\mathcal{H}\rightarrow 2^{\mathcal{H}}$ to indicate that $\mathcal{A}$ is a set-valued operator on $\mathcal{H}$ and the notation $\mathcal{A}:\mathcal{H}\rightarrow \mathcal{H}$ to indicate that $\mathcal{A}$ is a single-valued operator on $\mathcal{H}$. Given $\mathcal{A}:\mathcal{H}\rightarrow 2^{\mathcal{H}}$, its domain is denoted by $\mbox{ dom } \mathcal{A}:=\{x\in \mathcal{H}|\mathcal{A}x\neq \emptyset\}$, its graph is denoted by $\mbox{ gra }\mathcal{A}:=\{(x,u)\in \mathcal{H}\times \mathcal{H}|u\in \mathcal{A}x\}$, its set of zeros by $\mbox{ zer } \mathcal{A}:=\{x\in \mathcal{H}|0\in \mathcal{A}x\}$, its set of fixed points by $\mbox{ Fix } \mathcal{A}:=\{x\in \mathcal{H}|x\in \mathcal{A}x\}$, its resolvent with index $\gamma>0$ is denoted by $J_{\gamma}^{\mathcal{A}}:=(Id+\gamma \mathcal{A})^{-1}$ and its Yosida regularization with parameter $\gamma>0$ is denoted by $\mathcal{A}_{\gamma}:=\frac{1}{\gamma}(Id-J_{\gamma}^{\mathcal{A}})$.

\begin{definition}(See \cite{BauschkeC}) Let $T:\mathcal{H}\rightarrow \mathcal{H}$  and $\theta\in (0,1)$. 
 \begin{itemize}
 \item[(i)] $T$ is nonexpansive if $\|Tx-Ty\|\leq \|x-y\|$, $\forall (x,y)\in \mathcal{H}\times \mathcal{H}$.
 \item[(ii)] $T$ is $\theta-$averaged if there exists a nonexpansive operator $N:\mathcal{H}\rightarrow \mathcal{H}$ such that $T=(1-\theta)Id+\theta N$; equivalently, $\forall (x,y)\in \mathcal{H}\times \mathcal{H}$,
 $$(1-\theta)\|(Id-T)x-(Id-T)y\|^2\leq \theta(\|x-y\|^2-\|Tx-Ty\|^2).$$
 \item[(iii)] Let $\beta\in \mathbb{R}_{++}$. $T$ is $\beta-$cocoercive if 
 $$\langle x-y,Tx-Ty\rangle \geq \beta \|Tx-Ty\|^2, \forall(x,y)\in \mathcal{H}\times \mathcal{H}.$$	
 \end{itemize}
\end{definition}
 
\begin{definition} (See \cite[Definition 2.3] {BauschkeMW}) \label{comonotone}
 Let $\mathcal{A}:\mathcal{H}\rightarrow 2^{\mathcal{H}}$ and $\rho\in \mathbb{R}$. Then 
 \begin{itemize}
 \item[(i)] $\mathcal{A}$ is $\rho-$monotone if $\forall (x,u)\in \mbox{gra} \mathcal{A}$, $\forall (y,v)\in \mbox{gra }\mathcal{A}$, we have
 $$\langle x-y,u-v\rangle \geq \rho \|x-y\|^2.$$
 \item[(ii)] $\mathcal{A}$ is maximally $\rho-$monotone if $\mathcal{A}$ is $\rho-$monotone and there is no other $\rho-$monotone operator $\mathcal{B}:\mathcal{H}\rightarrow 2^{\mathcal{H}}$ such that $\text{gra} \mathcal{B}$ properly contains $\text{gra}\mathcal{A}$, $i.e.$, for every $(x,u)\in \mathcal{H} \times \mathcal{H}$,
 $$(x,u)\in \mbox{gra}\mathcal{A}\Leftrightarrow (\forall(y,v)\in \mbox{gra}\mathcal{A}) \langle x-y,u-v\rangle \geq \rho \|x-y\|^2.$$
 \item[(iii)] $\mathcal{A}$ is $\rho-$comonotone if $\forall (x,u)\in \mbox{gra} \mathcal{A}$, $\forall (y,v)\in \mbox{gra}\mathcal{A}$, 
 $$\langle x-y,u-v\rangle \geq \rho \|u-v\|^2.$$
 \item[(iv)] $\mathcal{A}$ is maximally $\rho-$comonotone if $\mathcal{A}$ is $\rho-$comonotone and there is no other $\rho-$comonotone operator $\mathcal{B}:\mathcal{H}\rightarrow 2^{\mathcal{H}}$ such that gra$\mathcal{B}$ properly contains gra$\mathcal{A}$, $i.e.$, for every $(x,u)\in \mathcal{H} \times \mathcal{H}$,
 $$(x,u)\in \mbox{gra}\mathcal{A}\Leftrightarrow (\forall(y,v)\in \mbox{gra}\mathcal{A}) \langle x-y,u-v\rangle \geq \rho \|u-v\|^2.$$	
 \end{itemize}
\end{definition}

\begin{remark} (See \cite[Remark 2.4]{BauschkeMW})\label{hypo}
\begin{itemize}
\item[(i)] When $\rho=0$, both $\rho-$monotonicity of $\mathcal{A}$ and $\rho-$comonotonicity of $\mathcal{A}$ reduce to the monotonicity of $\mathcal{A}$; equivalently to the monotonicity of ${\mathcal{A}}^{-1}$.
\item[(ii)] When $\rho<0$, $\rho-$monotonicity is know as $\rho-$hypomonotonicity, see \cite[Example 12.28] {RockafellarWets} and \cite [Definition 6.9.1] {Burachik}. In this case, the $\rho-$comonotonicity is also known as $\rho-$cohypomonotonicity (see \cite[Definition2.2]{CP}).
\item[(iii)] In passing, we point out that when $\rho>0$, $\rho-$monotonicity of  $\mathcal{A}$ reduces to $\rho-$strong monotonicity of $\mathcal{A}$, while $\rho-$comonotonicity of $\mathcal{A}$ reduces to $\rho-$cocoercivity of $\mathcal{A}$.
\end{itemize}	
\end{remark}

\begin{proposition} (See \cite [Proposition 3.7] {BauschkeMW}) \label{averaged}
Let $D$ be a nonempty subset of $\mathcal{H}$, $T: D\rightarrow \mathcal{H}$,  and $ \theta\in (0,1)$. Set $\mathcal{A}=T^{-1}-Id$, $i.e.$, $T=J^{\mathcal{A}}$, and $\rho=\frac{1}{2\theta}-1$. Then the following  conclusions hold:
\begin{itemize}
\item[(i)] $T$ is $\theta-$averaged $\Leftrightarrow$ $\mathcal{A}$ is $\rho-$comonotone.
\item[(ii)] $T$ is $\theta-$averaged and $D=\mathcal{H}$	 $\Leftrightarrow$ $\mathcal{A}$ is maximally $\rho-$comonotone.
\end{itemize}	
\end{proposition}

\begin{remark}\label{av}
Set $\gamma \mathcal{A}=T^{-1}-Id$, $i.e.$, $T=J_{\gamma}^{\mathcal{A}}$ and set $\rho=(\frac{1}{2\theta}-1)\gamma>-\frac{\gamma}{2}$, where $\gamma>0$. Then we have: $T$ is $\theta-$averaged $\Leftrightarrow$ $\mathcal{A}$ is $\rho-$comonotone. Furthermore, the crucial relation between $\mathcal{A}$ and $J_{\gamma}^{\mathcal{A}}$ is that the set of zeros of $\mathcal{A}$ coincides with the fixed point set of $J_{\gamma}^{\mathcal{A}}$ (which, therefore, in particular does not depend on the choice of $\gamma> 0$). Therefore $\mbox{zer} \mathcal{A}=\mbox{Fix} J_{\gamma}^{\mathcal{A}}=\mbox{zer} \mathcal{A}_{\gamma}$(See \cite [Proposition 23.2] {BauschkeC}  and \cite [Proposition 2.13] {BauschkeMW}).
\end{remark}

\begin{proposition}(See \cite [Proposition 3.7] {BartzD}) \label{single}
Let $\mathcal{A}:\mathcal{H}\rightarrow 2^{\mathcal{H}}$ be $\rho-$comonotone and let $\gamma\in \mathbb{R}_{++}$ such that $\gamma+\rho>0$. Then
\begin{itemize}
\item[(i)] $J_{\gamma}^{\mathcal{A}}$ is single-valued.
\item[(ii)]  $\mbox{dom} J_{\gamma}^{\mathcal{A}}=\mathcal{H}$ if and only if $\mathcal{A}$ is maximally $\rho-$comonotone.	
\end{itemize}
\end{proposition}

It is worth noting that from Remark $\ref{av}$ and Proposition $\ref{single}$ we can reasonably assume that for $\gamma>\max\{-2\rho,0\}$, the resolvent $J_{\gamma}^{\mathcal{A}}$ and the Yosida regularization $\mathcal{A}_{\gamma}$ are single-valued and averaged where $\mathcal{A}$ is $\rho-$comonotone.

\begin{proposition}\label{z}
Let $\mathcal{A}:\mathcal{H}\rightarrow 2^{\mathcal{H}}$ be $\rho-$comonotone, $\gamma>\max\{-2\rho,0\}$, and $\delta>\max\{-2\rho,0\}$. For any $x,y \in \mathcal{H}$, the following conclusions hold:
\begin{itemize}
\item[(i)] $J_{\gamma}^{\mathcal{A}}: \mathcal{H}\rightarrow \mathcal{H}$ is $\frac{\gamma}{2(\rho+\gamma)}-$averaged and ${\mathcal{A}}_{\gamma}:  \mathcal{H}\rightarrow \mathcal{H}$ is $\frac{2}{\gamma}-$Lipschitz continuous.
\item[(ii)] $\eta\in J_{\gamma}^{\mathcal{A}}x \Leftrightarrow (\eta, {\gamma}^{-1}(x-\eta))\in \mbox{gra}\mathcal{A}$.
\item[(iii)] $\mathcal{A}_{\gamma+\delta}=(\mathcal{A}_{\gamma})_{\delta}$.
\item[(iv)] $\mathcal{A}$ is maximally $\rho-$comonotone $\Leftrightarrow$ $\mathcal{A}_{\gamma}$ is $(\rho+\gamma)-$cocoercive. 
\item[(v)] $\mathcal{A}$ is maximally $\rho-$comonotone $\Leftrightarrow$ ${\mathcal{A}}^{-1}-\rho Id$ is maximally monotone.
\end{itemize} 
\end{proposition}
\begin{proof}
From Remark $\ref{av}$, $J_{\gamma}^{\mathcal{A}}$ is $\frac{\gamma}{2(\rho+\gamma)}-$averaged. 
We use the definition of the Yosida approximation and the average operator property of the resolvent to obtain
\begin{eqnarray}\label{Yosida}
\|\mathcal{A}_{\gamma}x-\mathcal{A}_{\gamma}y\|&=&\frac{1}{\gamma}\|(x-y)-(J_{\gamma}^{\mathcal{A}}x-J_{\gamma}^{\mathcal{A}}y)\|\nonumber\\
&\leq& \frac{1}{\gamma}\|x-y\|+\frac{1}{\gamma}\|J_{\gamma}^{\mathcal{A}}x-J_{\gamma}^{\mathcal{A}}y)\|\nonumber\\
&\leq& \frac{2}{\gamma}\|x-y\|.
\end{eqnarray}
$(ii)$ is from \cite[Proposition 23.2-(ii)]{BauschkeC}. This is true for $\mathcal{A}$ being a general operator.
$(iii)$ is from \cite[Proposition 23.7-(iii)]{BauschkeC}. This is true for $\mathcal{A}$ being a general operator.
$(iv)$ is from \cite[Lemma 3.2]{APR}, where $\gamma-\alpha=-\rho$.\\
$(v)$ is from \cite[Lemma 2.8]{BauschkeMW}.
\end{proof}

\begin{remark}\label{tuidao}
Let $y=x+\mathcal{A}_{\eta}x$. Then $x$ can be expressed in the following form
$$x=(1-\frac{1}{\eta+1})y+\frac{1}{\eta+1}J_{\eta+1}^{\mathcal{A}}y.$$	
\end{remark}
Indeed, according to the definition of $\mathcal{A}_{\eta}$ and $J_{\eta+1}^{\mathcal{A}}$,
\begin{eqnarray*}
&&y=x+\mathcal{A}_{\eta}x\\
&\Leftrightarrow & x=(Id+\mathcal{A}_{\eta})^{-1}y\\
&\Leftrightarrow & x=J_{1}^{\mathcal{A}_{\eta}}y\\
&\Leftrightarrow & x=[Id-(Id-J_{1}^{\mathcal{A}_{\eta}})]y\\
&\Leftrightarrow & x=[Id-(\mathcal{A}_{\eta})_1]y\\
&\overset{\text{\mbox{Proposition}\ref{z}-(iii)}}{\Leftrightarrow} & x=[Id-(\mathcal{A}_{\eta+1})]y\\
&\Leftrightarrow & x-y=-\mathcal{A}_{\eta+1}y\\
&\Leftrightarrow & x-y=-\frac{y-J_{\eta+1}^{\mathcal{A}}y}{\eta+1}\\
&\Leftrightarrow & x=(1-\frac{1}{\eta+1})y+\frac{1}{\eta+1}J_{\eta+1}^{\mathcal{A}}y.
\end{eqnarray*}	

\begin{proposition}\label{maximal}
Let $\mathcal{A}:\mathcal{H}\rightarrow 2^{\mathcal{H}}$ be maximally $\rho-$comonotone. For every sequence $(x_n,u_n)_{n\in \mathbb{N}}$ in gra$\mathcal{A}$ and every $(x,u)\in \mathcal{H} \times \mathcal{H}$, if $x_n\rightharpoonup x$ and $u_n\longrightarrow u$, then $(x,u)\in \mbox{gra} \mathcal{A}$.
\begin{proof}
Since  $\mathcal{A}$ is $\rho-$comonotone, for $\forall (x_n,u_n)\in \mbox{gra}\mathcal{A}$ and $\forall (y,v)\in \mbox{gra}\mathcal{A}$, we have $\langle x_n-y, u_n-v \rangle\geq \rho\|u_n-v\|^2$, which implies
\begin{eqnarray}\label{dayu}
\langle x_n-\rho u_n-y+\rho v, u_n-v \rangle\geq 0.	
\end{eqnarray}
In addition, since $(x_n,u_n)\in \mbox{gra}\mathcal{A}$ and $(y,v)\in \mbox{gra}\mathcal{A}$, we have $(u_n,x_n-\rho u_n)\in \mbox{gra} ({\mathcal{A}}^{-1}-\rho Id)$ and $(v,y-\rho v)\in \mbox{gra} ({\mathcal{A}}^{-1}-\rho Id)$. According to Proposition $\ref{z}-(v)$, we know ${\mathcal{A}}^{-1}-\rho Id$  is maximal monotone. Applying \cite[ Proposition 20.38]{BartzD} to $(\ref{dayu})$, we get 
$$\langle x-\rho u-y+\rho v, u-v \rangle\geq 0,$$
equivalently,
$$\langle x-y, u-v \rangle\geq \rho\|u-v\|^2.$$
Since the above inequality is true for $\forall (y,v)\in \mbox{gra}\mathcal{A}$, we conclude that $(x,u)\in \mbox{gra}\mathcal{A}$.
\end{proof}	
\end{proposition}

\begin{proposition}\label{der}
Let $\mathcal{A}:\mathcal{H}\rightarrow 2^{\mathcal{H}}$ be maximally $\rho-$comonotone, $\eta>\max\{-2\rho,0\}$, and $x(t)$ is a differentiable function. Set $z(t)=J_{\eta}^{\mathcal{A}}x(t)$. Then,
$$\|\dot{z}(t)\|\leq\|\dot{x}(t)\|.$$	
\end{proposition}
\begin{proof}
By Proposition $\ref{z}$-$(ii)$, we know $\mathcal{A}_{\eta}x(t)\in \mathcal{A}(z(t))$, $\mathcal{A}_{\eta}x(s)\in \mathcal{A}(z(s))$. In view of the $\rho-$comonotonicity of $\mathcal{A}$, we have 
$$\langle z(s)-z(t), \frac{x(s)-z(s)}{\eta}- \frac{x(t)-z(t)}{\eta}\rangle\geq\rho \|\frac{x(s)-z(s)}{\eta}- \frac{x(t)-z(t)}{\eta}\|^2.$$
Equivalently, 
$$\langle z(s)-z(t), x(s)-x(t)-(z(s)-z(t))\rangle\geq\frac{\rho}{\eta} \|x(s)-x(t)-(z(s)-z(t))\|^2.$$
This yields
\begin{eqnarray*}
&&\langle z(s)-z(t), x(s)-x(t)\rangle-\|z(s)-z(t)\|^2\\
&\geq&\frac{\rho}{\eta} \|x(s)-x(t)\|^2+\frac{\rho}{\eta} \|z(s)-z(t)\|^2\\
&&-\frac{2\rho}{\eta} \langle z(s)-z(t), x(s)-x(t)\rangle,
\end{eqnarray*}
which implies 
\begin{eqnarray*}
&&\frac{\rho}{\eta}\|x(s)-x(t)\|^2+(1+\frac{\rho}{\eta})\|z(s)-z(t)\|^2\\
&\leq&(1+\frac{2\rho}{\eta})\langle z(s)-z(t), x(s)-x(t)\rangle\\
&\leq&\frac{1+\frac{2\rho}{\eta}}{2}\|z(s)-z(t)\|^2+\frac{1+\frac{2\rho}{\eta}}{2}\|x(s)-x(t)\|^2,
\end{eqnarray*}
where $1+\frac{2\rho}{\eta}>0$ is due to $\eta>\max\{-2\rho,0\}$.\\
After simplification, we get the following formula
$$\|z(s)-z(t)\|^2\leq \|x(s)-x(t)\|^2.$$
Dividing by $s-t$ with $s\neq t$, and letting $s$ tend to $t$, we deduce that
$$\|\dot{z}(t)\|\leq \|\dot{x}(t)\|,$$
where according to Proposition $\ref{z}$-$(i)$, we know the resolvent $J_{\eta}^{\mathcal{A}}$	 is nonexpansive, thus, $z(t)$ is differentiable almost everywhere with respect to $t$.
\end{proof}

\begin{definition} (See \cite [Definition 2] {Bot})
We say that $x:[t_0,+\infty)\longrightarrow \mathcal{H}$ is a strong global solution of $(\ref{DS})$ if the following properties are satisfied:
\begin{itemize}
\item[(i)] $x$, $\dot{x}:[t_0,+\infty)\longrightarrow \mathcal{H}$ are locally absolutely continuous, in other words, absolutely continuous on each interval $[t_0,b]$ for $0<b<+\infty$;
\item[(ii)] $\mu(t)=\mathcal{A}_{\eta}x(t)$ and $\ddot{x}(t)+\dot{\mu}(t)+\frac{\alpha}{t}\dot{x}(t)+\frac{\beta}{t}\mu(t)=0$ for almost every $t\in[t_0,+\infty)$;
\item[(iii)] $x(t_0)=x_0$ and $\dot{x}(t_0)=v_0$.	
\end{itemize}	
\end{definition}

\begin{lemma} (See \cite [Proposition 6.2.1] {Haraux})\label{uniq}
Let $\mathbb{X}$ be a Banach space and $f:[t_0,+\infty) \times \mathbb{X}\rightarrow \mathbb{X}$	be a function. Suppose $f$ satisfies the following property:
\begin{itemize}
\item[(i)] $f(t,\cdot):\mathbb{X}\rightarrow\mathbb{X}$ is continuous and 
$$\|f(t,x)-f(t,y)\|\leq M(t,\|x\|+\|y\|)\|x-y\|,~~\forall x,y \in \mathbb{X},$$
$$M(t,r)\in L_{loc}^{1}([t_0,+\infty)),\forall r\in \mathbb{R}_{+};$$
for almost all $t\in [t_0,+\infty)$, where $L_{loc}^{1}([t_0,+\infty))$ denotes the family of locally integrable functions on $[t_0,+\infty)$;
\item[(ii)] for every $x\in \mathbb{X}$, $f(t,x)\in L_{loc}^{1}([t_0,+\infty))$;
\item[(iii)] $f(t,\cdot): \mathbb{X}\rightarrow \mathbb{X}$ satisfies
$$\|f(t,x)\|\leq P(t)(1+\|x\|)~~\mbox{and}~~P(t)\in L_{loc}^{1}([t_0,+\infty))$$
for almost all $t\in [t_0,+\infty)$.
\end{itemize}
Then, for
$$\frac{d}{dt}x(t)=f(t,x(t)),~~~~x(t_0)=x_0,$$
there exists a unique global trajectory $x: [t_0, +\infty)\rightarrow \mathbb{X}$. 
\end{lemma}

\begin{lemma} (See \cite [Lemma A.5] {AttouchP}) \label{w}
Let $\omega, \eta:[t_0,+\infty)\rightarrow [0,+\infty)$ be absolutely continuous functions such that $\eta\notin L^{1}(t_0,+\infty)$,
$$\int_{t_0}^{+\infty}\omega(t)\eta(t)dt<+\infty,$$
and $|\dot{\omega}(t)|\leq \eta(t)$ for almost every $t>t_0$. Then $\lim_{t\rightarrow +\infty}\omega(t)=0$. 
\end{lemma}
\begin{lemma} (See \cite [Lemma 5.1] {NL}) \label{exists}
Suppose that $F:[0,+\infty)\rightarrow \mathbb{R}$ is locally absolutely continuous and bounded below and that there exist $G\in L^{1}([0,+\infty)$ such that for almost every $t\in [0,+\infty)$
$$\frac{d}{dt}F(t)\leq G(t).$$
Then there exists $\lim_{t\rightarrow +\infty}F(t)\in \mathbb{R}$.	
\end{lemma}

\begin{lemma}\label{op-l} (See \cite [Lemma A.2] {AttouchPR})\label{x-x}
Let $\mathcal{H}$ be a Hilbert space. Let $x:[t_0, +\infty)\rightarrow \mathcal{H}$ a continuously differentiable function satisfying $q(t)+\frac{t}{\sigma}\dot{q}(t)\rightarrow L$, $t\rightarrow +\infty$, with $\sigma>0$ and $L\in \mathcal{H}$. Then $q(t)\rightarrow L$, $t\rightarrow +\infty$. 
\end{lemma}
\begin{lemma} (See \cite [Lemma 5.3] {NL}) \label{OP}
Let $S$ be a nonempty subset of $\mathcal{H}$, and $x:[t_0,+\infty)\rightarrow \mathcal{H}$ a map. Assume that 
\begin{itemize}
\item[(i)] for every $z\in S$, $\lim\limits_{t\rightarrow +\infty}\|x(t)-z\|$ exists;
\item[(ii)]	every weak sequential limit point of  $x(t)$, as $t\rightarrow +\infty$, belongs to $S$.
\end{itemize}
Then $x(t)$ converges weakly as $t\rightarrow +\infty$ to a point in $S$.	
\end{lemma}

The discrete version of Lemma \ref{op-l}  is often referred to as Opial's Lemma \cite{Opial}.

\begin{lemma} \label{op}
Let $S$ be a nonempty subset of $\mathcal{H}$, and $\{x_k\}$ a sequence of elements of $\mathcal{H}$. Assume that
\begin{itemize}
\item[(i)] for every $z\in S$, $\lim\limits_{k\rightarrow +\infty}\|x_k-z\|$ exists;
\item[(ii)] every weak sequential limit point of $\{x_k\}$, as $k\rightarrow +\infty$, belongs to $S$.
\end{itemize}
Then $x_k$ converges weakly as $k\rightarrow +\infty$ to a point in $S$.	
\end{lemma}

\begin{lemma} (See \cite [Fact 2.5] {Ou}) \label{des}
Let $\{a_k\}_{k\in \mathbb{N}}$ and $\{b_k\}_{k\in \mathbb{N}}$ be sequences in $\mathbb{R}_{+}$ such that $\sum_{k\in \mathbb{N}}b_k< +\infty$ and 
$$(\forall \in \mathbb{N}) ~~~~a_{k+1}\leq a_k+b_k.$$	
Then $\lim_{k\rightarrow +\infty}a_k\in \mathbb{R}_{+}$.
\end{lemma}

\section{Existence and uniqueness of solutions}

Let us first establish the existence and uniqueness of the solution trajectory of the Cauchy problem associated to the continuous dynamical system $(\ref{DS})$. In what follows, we always suppose that
$$ ~~~~~~~\mathcal{A} ~\mbox{is}~\rho-\mbox{comonotone}\text{ with }\rho\in \mathbb{R},~~\mbox{zer} \mathcal{A}\neq\emptyset~~\mbox{and}~~ \eta>\max\{-2\rho,0\}.~~~~~~~~~\quad(H)$$
\begin{theorem}\label{theorem1}
Under $(H)$, take $t_0>0$. Then, for any $x_0\in \mathcal{H}$, $x_1 \in \mathcal{H}$, there exists a unique strong global solution $x: [t_0, +\infty) \rightarrow \mathcal{H}$ of the dynamical system $(\ref{DS})$ which satisfies the Cauchy data $x(t_0)=x_0$ and $\dot{x}(t_0)=x_1$.
\end{theorem}
\begin{proof}
First rewrite the system $(\ref{DS})$ as follows
\begin{equation*}
\left\{
\begin{array}{lcl}
y(t)=-\dot{x}(t)-\mathcal{A}_{\eta}x(t)\\
\dot{y}(t)=\frac{\alpha}{t}\dot{x}(t)+\frac{\beta}{t}\mathcal{A}_{\eta}x(t),
\end{array}
\right.
\end{equation*}	
equivalently,
\begin{equation*}
\left\{
\begin{array}{lcl}
\dot{x}(t)=-\mathcal{A}_{\eta}x(t)-y(t)\\
\dot{y}(t)=\frac{\beta-\alpha}{t}\mathcal{A}_{\eta}x(t)-\frac{\alpha}{t}y(t).
\end{array}
\right.
\end{equation*}
Hence, the system $(\ref{DS})$ can be equivalently written as a first-order dynamical system in the phase space $\mathcal{H}\times\mathcal{H}$ with the Cauchy data $x(t_0)=x_0$ and $y(t_0)=y_0:=-x_1-\mathcal{A}_{\eta}x_0$
\begin{equation*}
\left\{
\begin{array}{lcl}
\dot{Z}(t)=F(t,Z(t))\\
Z(t_0)=(x_0,y_0),
\end{array}
\right.
\end{equation*}	
with
$$Z:[t_0,+\infty)\rightarrow \mathcal{H}\times\mathcal{H}, Z(t)=(x(t),y(t))$$
and
$$F:[t_0,+\infty)\times \mathcal{H}\times\mathcal{H}\rightarrow \mathcal{H}\times\mathcal{H}, F(t,x,y)=(-\mathcal{A}_{\eta}x-y,\frac{\beta-\alpha}{t}\mathcal{A}_{\eta}x-\frac{\alpha}{t}y).$$
We endow $\mathcal{H}\times \mathcal{H}$ with scalar product $\langle (x,y),(\bar{x},\bar{y})\rangle_{\mathcal{H}\times\mathcal{H}}=\langle x,\bar{x}\rangle+\langle y,\bar{y}\rangle$ and corresponding norm $\|(x,y)\|_{\mathcal{H}\times \mathcal{H}}=\sqrt{\|x\|^2+\|y\|^2}$.\\

{\it Step 1:} For arbitrary $(x,y),~(\bar{x},\bar{y})\in \mathcal{H}\times \mathcal{H}$, we make the following estimation
\begin{eqnarray*}
&&\|F(t,x,y)-F(t,\bar{x},\bar{y})\|_{\mathcal{H}\times \mathcal{H}}\\
&=&\sqrt{\|-\mathcal{A}_{\eta}x-y+\mathcal{A}_{\eta}\bar{x}+\bar{y}\|^2+\|\frac{\beta-\alpha}{t}\mathcal{A}_{\eta}x-\frac{\alpha}{t}y-\frac{\beta-\alpha}{t}\mathcal{A}_{\eta}\bar{x}+\frac{\alpha}{t}\bar{y}\|^2}\\
&\leq&\sqrt{2\|\mathcal{A}_{\eta}x-\mathcal{A}_{\eta}\bar{x}\|^2+2\|y-\bar{y}\|^2+2\frac{(\alpha-\beta)^2}{t^2}\|\mathcal{A}_{\eta}x-\mathcal{A}_{\eta}\bar{x}\|^2+2\frac{{\alpha}^2}{t^2}\|y-\bar{y}\|^2}\\
&=&\sqrt{(2+2\frac{(\alpha-\beta)^2}{t^2})\|\mathcal{A}_{\eta}x-\mathcal{A}_{\eta}\bar{x}\|^2+(2+2\frac{{\alpha}^2}{t^2})\|y-\bar{y}\|^2}\\
&\overset{\text{(\ref{Yosida})}}{\leq}&\sqrt{(2+2\frac{(\alpha-\beta)^2}{t^2})\frac{4}{{\eta}^2}\|x-\bar{x}\|^2+(2+2\frac{{\alpha}^2}{t^2})\|y-\bar{y}\|^2}\\
&\leq&\sqrt{(2+2\frac{(\alpha-\beta)^2}{t^2})\frac{4}{{\eta}^2}+2+\frac{2{\alpha}^2}{t^2}}\|(x,y)-(\bar{x},\bar{y})\|\\
&\leq&(\frac{2\sqrt{2}}{\eta}+\sqrt{2}+\frac{2\sqrt{2}|\alpha-\beta|}{t\eta}+\frac{\sqrt{2}\alpha}{t})\|(x,y)-(\bar{x},\bar{y})\|_{\mathcal{H}\times \mathcal{H}}.
\end{eqnarray*}
By employing the notation $N(t):=\frac{2\sqrt{2}}{\eta}+\sqrt{2}+\frac{2\sqrt{2}|\alpha-\beta|}{t\eta}+\frac{\sqrt{2}\alpha}{t},~~\forall t\in [t_0,+\infty)$, we conclude that
$$\|F(t,x,y)-F(t,\bar{x},\bar{y})\|_{\mathcal{H}\times \mathcal{H}}\leq N(t)\|(x,y)-(\bar{x},\bar{y})\|_{\mathcal{H}\times \mathcal{H}}.$$
Hence $F(t,\cdot,\cdot)$ is $N(t)-$Lipschitz continuous for every $t\geq t_0$.
Moreover, for any $t\geq t_0$, by the continuity of $\frac{1}{t}$, we know that $N(\cdot)$ is integrable on $[t_0,T]$ for any $t_0<T<+\infty$. Thus $N(\cdot)\in L_{loc}^{1}([t_0,+\infty)$($N(t)$ is said to be locally integrable in the interval $0\leq t<+\infty$ if it is integrable (in the sense of Lebesgue) in each bounded interval $0\leq t\leq T$).
As $N(\cdot)\in L_{loc}^{1}([t_0,+\infty)$, the Lipschitz constant of $F(t,\cdot,\cdot)$ is local integrable.

{\it Step 2:} For fixed $x,y\in \mathcal{H}$, it holds
\begin{eqnarray*}
&&\int_{t_0}^{T}\|F(t,x,y)\|_{\mathcal{H}\times \mathcal{H}} dt\\
&=& \int_{t_0}^{T}\sqrt{\|\mathcal{A}_{\eta}x+y\|^2+\|\frac{\alpha-\beta}{t}\mathcal{A}_{\eta}x+\frac{\alpha}{t}y\|^2} dt\\
&\leq&\int_{t_0}^{T}\sqrt{2\|y\|^2+\frac{{2\alpha}^2}{t^2}\|y\|^2+(2+\frac{2(\alpha-\beta)^2}{t^2})\|\mathcal{A}_{\eta}x\|^2} dt\\
&\leq&\int_{t_0}^{T}[(\sqrt{2}+\frac{\sqrt{2}\alpha}{t})\|y\|+(\sqrt{2}+\frac{\sqrt{2}|\alpha-\beta|}{t})\|\mathcal{A}_{\eta}x\|] dt.
\end{eqnarray*}
The continuity of $\frac{1}{t}$ yields
$$\int_{t_0}^{T}\|F(t,x,y)\|_{\mathcal{H}\times \mathcal{H}} dt<+\infty,~~~~\forall t_0<T<+\infty.$$

{\it Step 3:} For arbitrary $x^*\in \mbox{zer}{\mathcal{A}}$, by employing Remark \ref{av} and $(\ref{Yosida})$, we consider the following estimation
\begin{eqnarray*}
&&\|F(t,x,y)\|_{\mathcal{H}\times \mathcal{H}}\\
&=& \sqrt{\|\mathcal{A}_{\eta}x+y\|^2+\|\frac{\alpha-\beta}{t}\mathcal{A}_{\eta}x+\frac{\alpha}{t}y\|^2}\\
&\leq& \sqrt{(2+\frac{2{\alpha}^2}{t^2})\|y\|^2+(2+\frac{2(\alpha-\beta)^2}{t^2})\|\mathcal{A}_{\eta}x\|^2}\\
&=& \sqrt{(2+\frac{2{\alpha}^2}{t^2})\|y\|^2+(2+\frac{2(\alpha-\beta)^2}{t^2})\|\mathcal{A}_{\eta}x-\mathcal{A}_{\eta}x^*\|^2}\\
&\leq& \sqrt{(2+\frac{2{\alpha}^2}{t^2})\|y\|^2+(2+\frac{2(\alpha-\beta)^2}{t^2})\frac{4}{{\eta}^2}\|x-x^*\|^2}\\
&\leq& \sqrt{(2+\frac{2{\alpha}^2}{t^2})\|y\|^2+(2+\frac{2(\alpha-\beta)^2}{t^2})\frac{8}{{\eta}^2}\|x\|^2+(2+\frac{2(\alpha-\beta)^2}{t^2})\frac{8}{{\eta}^2}\|x^*\|^2}\\
&\leq& (\frac{4}{\eta}+\frac{4|\alpha-\beta|}{t\eta})\|x^*\|+(\sqrt{2}+\frac{\sqrt{2}\alpha}{t}+\frac{4}{\eta}+\frac{4|\alpha-\beta|}{t\eta})\|(x,y)\|_{\mathcal{H}\times \mathcal{H}}\\
&\leq& P(t)(1+\|(x,y)\|_{\mathcal{H}\times \mathcal{H}}),
\end{eqnarray*}
where $P(t)=(\frac{4}{\eta}+\frac{4|\alpha-\beta|}{t\eta})\|x^*\|+\sqrt{2}+\frac{\sqrt{2}\alpha}{t}+\frac{4}{\eta}+\frac{4|\alpha-\beta|}{t\eta}$. By virtue of the continuity of $\frac{1}{t}$ with respect to $t$, we conclude that $P(t)\in L_{loc}^{1}([t_0,+\infty)$. 
Based on the above statements, the existence and uniqueness of a strong global solution are consequences of the Cauchy-Lipschitz-Picard Theorem (Lemma \ref{uniq}) for the first order dynamical system.
From here, due to the equivalence of the first order dynamical system and $(\ref{DS})$, the conclusion follows.
\end{proof}

\section{Convergence analysis}
In this section we will analyze the convergence properties of the trajectories generated by the dynamical system  $(\ref{DS})$. 

\begin{theorem}\label{bounded}
Under $(H)$, let $x:[t_0,+\infty)\rightarrow \mathcal{H}$ be a trajectory of $(\ref{DS})$, where the parameters satisfy $\alpha\geq\beta+1$ and $\beta>1$. Then we have
\begin{itemize}
\item [(i)] $\int_{t_0}^{+\infty}t\|\dot{x}(t)\|^2dt<+\infty$, $\int_{t_0}^{+\infty} t\|\dot{x}(t)+\mu(t)\|^2 dt <+\infty$, $\int_{t_0}^{+\infty} t\|\mu(t)\|^2 dt <+\infty$ and $\int_{t_0}^{+\infty} t|\langle \dot{x}(t), \mu(t)\rangle| dt <+\infty$.
\item [(ii)] $\lim_{t\rightarrow+\infty}t\|\dot{x}(t)\|=0$, $\lim_{t\rightarrow+\infty}t\|\mu(t)\|=0$ and $\lim_{t\rightarrow+\infty}t\|\dot{x}(t)+\mu(t)\|=0$.
\item [(iii)] for any $x^*\in \mbox{zer} \mathcal{A}$, $\lim_{t\rightarrow+\infty}\|x(t)-x^*\|^2$ exists.
\item [(iv)] $x(\cdot)$ converges weakly, as $t\rightarrow +\infty$, to an element of $\mbox{zer} \mathcal{A}$.
\end{itemize}
\end{theorem}
\begin{proof}
Take $x^*\in \mbox{zer}\mathcal{A}$. For $0 \leq s \leq \alpha-1 $,  define
$$\varepsilon_s(t):=\frac{1}{2}\|s(x^*-x(t))-t\dot{x}(t)\|^2+\frac{s(\alpha-s-1)}{2}\|x(t)-x^*\|^2+st\langle x(t)-x^*,\mu(t) \rangle.$$
In view of system $(\ref{DS})$, we calculate its time derivative as
\begin{eqnarray}\label{Ly}
\frac{d\varepsilon_s(t)}{dt}&=&	\langle s(x^*-x(t))-t\dot{x}(t),- s\dot{x}(t)-\dot{x}(t)-t\ddot{x}(t)\rangle\nonumber\\
&&+s(\alpha-s-1)\langle x(t)-x^*,\dot{x}(t) \rangle+s\langle x(t)-x^*,\mu(t) \rangle\nonumber\\
&&+st\langle\dot{x}(t),\mu(t)\rangle+st\langle x(t)-x^*,\dot{\mu}(t) \rangle\nonumber\\
&=& s\langle x^*-x(t),(\alpha-s-1)\dot{x}(t)\rangle + st \langle x^*-x(t), \dot{\mu}(t)\rangle\nonumber\\
&&+ s\beta\langle x^*-x(t),\mu(t)\rangle -t(\alpha-s-1)\|\dot{x}(t)\|^2\nonumber\\
&&-t^2\langle \dot{x}(t),\dot{\mu}(t) \rangle-\beta t\langle \dot{x}(t),\mu(t) \rangle +s\langle x(t)-x^*,\mu(t)\rangle\nonumber\\
&&+st\langle\dot{x}(t), \mu(t)\rangle+st\langle x(t)-x^*, \dot{\mu}(t)\rangle \nonumber\\
&&+s(\alpha-s-1)\langle x(t)-x^*,\dot{x}(t)\rangle\nonumber\\
&=&s(\beta-1)\langle x^*-x(t),\mu(t)\rangle-t(\alpha-s-1)\|\dot{x}(t)\|^2\nonumber\\
&&-t^2\langle\dot{x}(t),\dot{\mu}(t) \rangle+(s-\beta)t\langle \dot{x}(t),\mu(t) \rangle.
\end{eqnarray}
Let $z(t)=J_{\eta}^{\mathcal{A}}x(t)$ and $\mu(t)=\mathcal{A}_{\eta}x(t)=\frac{x(t)-z(t)}{\eta}$. By Proposition \ref{der}, we conclude that
\begin{eqnarray}\label{erjie}
&&-\langle \dot{x}(t), \dot{\mu}(t)\rangle\nonumber\\
&=&-\frac{1}{\eta}\langle \dot{x}(t), \dot{x}(t)-\dot{z}(t)\rangle\nonumber\\
&=& -\frac{1}{\eta}\|\dot{x}(t)\|^2+\frac{1}{\eta}\langle \dot{x}(t), \dot{z}(t)\rangle\nonumber\\
&\leq&-\frac{1}{\eta}\|\dot{x}(t)\|^2+\frac{1}{\eta}\|\dot{x}(t)\|^2\nonumber\\
&=& 0.
\end{eqnarray}
From $(\ref{Ly})$,we find estimating $(s-\beta)t\langle \dot{x}(t), \mu (t) \rangle$ is tricky, while the rest of the right-hand side of the above inequality is nonpositive. We're going to focus on whether $s$ is equal to $\beta$.

Take $s=\beta$. Since $x^*\in \mbox{zer}\mathcal{A}=\mbox{zer}\mathcal{A}_{\eta}$ (see Remark $\ref{av}$) and according to Proposition $\ref{z}$-$(iv)$, $(\ref{Ly})$ yields
\begin{eqnarray*}
\frac{d\varepsilon_{\beta}(t)}{dt}
&=&({\beta}^2-\beta)\langle x^*-x(t), \mu(t)\rangle-t(\alpha-\beta-1)\|\dot{x}(t)\|^2-t^2\langle \dot{\mu}(t),\dot{x}(t) \rangle\\
&\leq&-({\beta}^2-\beta)(\rho+\eta)\|\mu(t)\|^2-t(\alpha-\beta-1)\|\dot{x}(t)\|^2-t^2\langle \dot{\mu}(t),\dot{x}(t) \rangle.
\end{eqnarray*}
Noticing that $\beta>1$, $0\leq s\leq \alpha-1$, $\eta>\max\{-2\rho,0\}$ and $(\ref{erjie})$, we obtain $\frac{d\varepsilon_{\beta}(t)}{dt}\leq 0$.
It follows that the function $\varepsilon_{\beta}(t)$ is nonincreasing as $t\rightarrow +\infty$. Since it is nonnegative,   $\varepsilon_{\beta}(t)$ has a limit as $t\rightarrow +\infty$, and the trajectory $x(t)$ is bounded. Integrate the above equality from $t_0$ to $+\infty$ to get
\begin{eqnarray*}
&&\lim_{t\rightarrow +\infty}\varepsilon_{\beta}(t)+({\beta}^2-\beta)\int_{t_0}^{+\infty}\langle x(t)-x^*, \mu(t)\rangle dt\\
&+&\int_{t_0}^{+\infty}t(\alpha-\beta-1)\|\dot{x}(t)\|^2dt +\int_{t_0}^{+\infty}t^2 \langle \dot{\mu}(t),\dot{x}(t) \rangle  dt\\
&=& \varepsilon_{\beta}(t_0),
\end{eqnarray*}
which implies
\begin{equation}\label{euql}
\int_{t_0}^{+\infty}\langle x(t)-x^*, \mu(t)\rangle dt <+\infty,
\end{equation}
$$\int_{t_0}^{+\infty} t^2 \langle \dot{\mu}(t),\dot{x}(t) \rangle  dt <+\infty,$$
\begin{equation}\label{keji}
\int_{t_0}^{+\infty} t\|\dot{x}(t)\|^2dt<+\infty.	
\end{equation}
Combining Proposition $\ref{z}$-$(iv)$ with $(\ref{euql})$, we have 
$$\int_{t_0}^{+\infty}\|\mu(t)\|^2dt <+\infty.$$
Clearly, according to the definition of $\varepsilon_{\beta}(t)$, we deduce that $\|x(t)-x^*\|^2$, $t\|\mu(t)\|^2$ and $\|s(x^*-x(t))-t\dot{x}(t)\|$ are bounded on $t\in [t_0,+\infty)$, because of the boundedness of $\varepsilon_{\beta}(t)$. Since $t\|\dot{x}(t)\|-s\|x(t)-x^*\|\leq\|s(x^*-x(t))-t\dot{x}(t)\|$, we know $t\|\dot{x}(t)\|$ is bounded  on $t\in [t_0,+\infty)$.
By Proposition \ref{der}, 
 \begin{eqnarray}\label{u}
 \|\dot{\mu}(t)\|=\frac{1}{\eta}\|\dot{x}(t)-\dot{z}(t)\|\leq\frac{2}{\eta}\|\dot{x}(t)\|,	
 \end{eqnarray}
so $t\|\dot{\mu}(t)\|$ is bounded.
According to dynamical system $(\ref{DS})$, $\|\ddot{x}(t)\|\leq\|\dot{\mu}(t)\|+\frac{\beta}{t}\|\mu(t)\|+\frac{\alpha}{t}\|\dot{x}(t)\|$.
Multiplying each member of the above inequality by $t$, we derive that
$t\|\ddot{x}(t)\|\leq t\|\dot{\mu}(t)\|+\beta\|\mu(t)\|+\alpha\|\dot{x}(t)\|.$ Thus, $t\|\ddot{x}(t)\|$ is bounded on $t\in [t_0,+\infty)$. This gives
\begin{eqnarray*}
&&|\frac{d(t\|\dot{x}(t)\|^2)}{dt}|\\
&=&\|\dot{x}(t)\|^2+2t|\langle \dot{x}(t),\ddot{x}(t)\rangle	|\\
&\leq& \|\dot{x}(t)\|^2+2t\|\dot{x}(t)\|\|\ddot{x}(t)\|\\
&<& M,
\end{eqnarray*}
where $M$ is a positive constant. 
Applying Lemma \ref{w} with $w(t)=t\|\dot{x}(t)\|^2$ and $\eta(t)=M$, we find
$$\lim_{t\rightarrow+\infty} t\|\dot{x}(t)\|^2=0.$$
This together with $(\ref{u})$ yields
$$\lim_{t\rightarrow+\infty}t\|\dot{\mu}(t)\|^2=0.$$

Take $s\neq \beta$. In order to estimate $(s-\beta)t\langle \dot{x}(t), \mu (t) \rangle$, considering the system $(\ref{DS})$, we observe that
\begin{eqnarray}\label{inte}
&&\frac{d(t^2\|\dot{x}(t)+\mu(t)\|^2)}{dt}\nonumber\\
&=&2t\|	\dot{x}(t)+\mu(t)\|^2+2t^2\langle \dot{x}(t)+\mu(t), \ddot{x}(t)+\dot{\mu}(t)\nonumber \rangle\\
&=&2t\|	\dot{x}(t)+\mu(t)\|^2+2t^2\langle \dot{x}(t)+\mu(t), -\frac{\alpha}{t}\dot{x}(t)-\frac{\beta}{t}\mu(t)\nonumber \rangle\\
&=&2t\|	\dot{x}(t)+\mu(t)\|^2-2t\langle \dot{x}(t)+\mu(t), \beta\dot{x}(t)+\beta\mu(t)+(\alpha-\beta)\dot{x}(t)\nonumber \rangle\\
&=&2(1-\beta)t\|\dot{x}(t)+\mu(t)\|^2-2(\alpha-\beta)t\langle \dot{x}(t)+\mu(t), \dot{x}(t)\nonumber\rangle\\
&\leq& 2(1-\beta)t\|\dot{x}(t)+\mu(t)\|^2+2(\alpha-\beta)t|\langle\dot{x}(t)+\mu(t), \dot{x}(t) \nonumber\rangle|\\
&\leq& 2(1-\beta)t\|\dot{x}(t)+\mu(t)\|^2+(\alpha-\beta)\xi t\|\dot{x}(t)+\mu(t)\|^2+ (\alpha-\beta)\frac{1}{\xi}t\|\dot{x}(t)\|^2\nonumber\\
&=& [2(1-\beta)+(\alpha-\beta)\xi]t\|\dot{x}(t)+\mu(t)\|^2+(\alpha-\beta)\frac{1}{\xi}t\|\dot{x}(t)\|^2.
\end{eqnarray} 
Taking into account the assume $\alpha\geq\beta+1$ and $\beta >1$, let us choose $\xi\in(0,\frac{2(\beta-1)}{\alpha-\beta})$, such that $2(1-\beta)+(\alpha-\beta)\xi <0$. By ignoring the nonpositive term $[2(1-\beta)+(\alpha-\beta)\xi]t\|\dot{x}(t)+\mu(t)\|^2$, $(\ref{inte})$ becomes $\frac{d(t^2\|\dot{x}(t)+\mu(t)\|^2)}{dt}\leq(\alpha-\beta)\frac{1}{\xi}t\|\dot{x}(t)\|$. Now, applying Lemma $\ref{exists}$ to this inequality, we deduce that $\lim_{t\rightarrow +\infty}t^2\|\dot{x}(t)+\mu(t)\|^2$ exists, because the right-hand side is integrable by $(\ref{keji})$. Integrate $(\ref{inte})$ from $t_0$ to $+\infty$ to obtain
 \begin{eqnarray*}
 &&\lim_{t_0\rightarrow +\infty}t^2\|\dot{x}(t)+\mu(t)\|^2+\int_{t_0}^{+\infty}[2(\beta-1)-(\alpha-\beta)\frac{1}{\xi}]t\|\dot{x}(t)+\mu(t)\|^2 dt\\
 &\leq&\int_{t_0}^{+\infty}(\alpha-\beta)\frac{1}{\xi}t\|\dot{x}(t)\|^2 dt+t_{0}^{2}\|\dot{x}(t_0)+\mu(t_0)\|^2\\
 &<&+\infty,
 \end{eqnarray*}
 which implies that
 \begin{equation}\label{xu}
 \int_{t_0}^{+\infty}t\|\dot{x}(t)+\mu(t)\|^2 dt <+\infty.	
 \end{equation}
 Further, this result and $(\ref{keji})$ leads to 
 \begin{eqnarray}\label{ukj}
 \int_{t_0}^{+\infty}t\|\mu(t)\|^2 dt <+\infty,	
 \end{eqnarray}
 because of 
  \begin{eqnarray*} 
 t\|\mu(t)\|^2&=&t\|\mu(t)+\dot{x}(t)-\dot{x}(t)\|^2\\
 &\leq&2t\|\mu(t)+\dot{x}(t)\|^2+2t\|\dot{x}(t)\|^2.
  \end{eqnarray*}
Observe that
$$t|\langle \dot{x}(t),\mu(t)\rangle|\leq\frac{1}{2}t\|\dot{x}(t)\|^2+\frac{1}{2}t\|\mu(t)\|^2$$
 and this together with $(\ref{keji})$ and $(\ref{ukj})$ leads to  
 $$\int_{t_0}^{+\infty}t|\langle \dot{x}(t),\mu(t)\rangle|dt <+\infty.$$
 Here we have completed the proof of $(i)$.\\
 Neglecting nonpositive terms of the right-hand side of $(\ref{Ly})$, we derive
 \begin{eqnarray*}
 \frac{d\varepsilon_s(t)}{dt}&\leq&(s-\beta)t\langle \mu(t),\dot{x}(t) \rangle\\
 &\leq&|s-\beta| t|\langle \dot{x}(t), \mu(t)\rangle|.	
 \end{eqnarray*}
Applying Lemma $\ref{exists}$ to the above inequality, and using the integrability of  $t|\langle \dot{x}(t),\mu(t)\rangle|$, it follows that $\lim_{t_0\rightarrow +\infty}\varepsilon_s (t)$ exists.

Collecting these results, we deduce that $\lim_{t_0\rightarrow +\infty}\varepsilon_s (t)$ exists for $s\in [0,\alpha-1]$, whether $s$ equals $\beta$ or not. By taking $s=0$, we find $\lim_{t_0\rightarrow +\infty}\varepsilon_0(t)=\frac{1}{2}t^2\|\dot{x}(t)\|^2$ exists. Furthermore, we claim that 
\begin{equation}\label{01}
\lim_{t_0\rightarrow +\infty}t\|\dot{x}(t)\|=0.
\end{equation}
Otherwise, if its limit is a non-zero constant, that would contradict $(\ref{keji})$.
 In the same way, combining the existence of $\|t\dot{x}(t)+t\mu(t)\|^2$ and $(\ref{xu})$, we have 
 \begin{equation}\label{02}
 \lim_{t_0\rightarrow +\infty}t\|\dot{x}(t)+\mu(t)\|=0.	
 \end{equation}
Using $(\ref{01})$ and $(\ref{02})$, we can improve the estimate of $\mu(t)$.\\
First observe that
\begin{eqnarray*}
&&\|t\dot{x}(t)+t\mu(t)\|^2\\
&=& t^2\|\dot{x}(t)\|^2+2t^2\langle\dot{x}(t),\mu(t)\rangle+t^2\|\mu(t)\|^2\\
&\geq& t^2\|\dot{x}(t)\|^2-4t^2\|\dot{x}(t)\|^2-\frac{1}{4}t^2\|\mu(t)\|^2+t^2\|\mu(t)\|^2\\
&=&-3t^2\|\dot{x}(t)\|^2+\frac{3}{4}t^2\|\mu(t)\|^2.
\end{eqnarray*}
equivalently, 
$$\frac{3}{4}t^2\|\mu(t)\|^2\leq t^2\|\dot{x}(t)+\mu(t)\|^2+3t^2\|\dot{x}(t)\|^2.$$
Next, combining this relation with $(\ref{01})$ and $(\ref{02})$, we infer that
\begin{equation}\label{sorder}
\lim_{t\rightarrow +\infty}t\|\mu(t)\|=0.	
\end{equation}
Here we have completed the proof of $(ii)$.

For any two distinct values $ s_1, s_2 \in [0,\alpha-1]$, let us take the definition of $\varepsilon_s(t)$ into account and expand the square to get
 \begin{eqnarray*}
&& \varepsilon_{s_1}(t)-\varepsilon_{s_2}(t)\\
 &=& \frac{1}{2}{s_{1}^{2}}\|x^*-x(t)\|^2-s_1 t\langle x^*-x(t),\dot{x}(t)+\mu(t)\rangle+\frac{s_1(\alpha-s_1-1)}{2}\|x(t)-x^*\|^2\\
 &-& \frac{1}{2}{s_{2}^{2}}\|x^*-x(t)\|^2+s_2 t\langle x^*-x(t),\dot{x}(t)+\mu(t)\rangle-\frac{s_2(\alpha-s_2-1)}{2}\|x(t)-x^*\|^2\\
 &=&(s_1-s_2)[t\langle\dot{x}(t)+\mu(t),x(t)-x^*\rangle+\frac{\alpha-1}{2}\|x(t)-x^*\|^2]. 	
 \end{eqnarray*}
 We deduce that the quantity $k(t)$, defined as
 $$k(t):=t\langle\dot{x}(t)+\mu(t),x(t)-x^*\rangle+\frac{\alpha-1}{2}\|x(t)-x^*\|^2$$
 has a limit as $t\rightarrow +\infty$, because of the existence of $\lim_{t\rightarrow +\infty}\varepsilon_{s}(t)$. Our goal is to show that each term has a limit. By setting 
 $$q(t)=\frac{1}{2}\|x(t)-x^*\|^2+\int_{t_0}^{t}\langle x(s)-x^*,\mu(s)\rangle ds,$$ 
 then $\dot{q}(t)=\langle x(t)-x^*, \dot{x}(t)+\mu(t)\rangle$, we may write $k(t)$ as
$$k(t)=t\dot{q}(t)+(\alpha-1)q(t)-(\alpha-1) \int_{t_0}^{t}\langle x(s)-x^*, \mu(s)\rangle ds.$$
By using $(\ref{euql})$, we know the last term $\int_{t_0}^{t}\langle x(s)-x^*, \mu(s)\rangle ds$ has a limit as $t\rightarrow +\infty$. This together with the  existence of $\lim_{t\rightarrow +\infty}k(t)$ yields $\lim_{t\rightarrow +\infty}t\dot{q}(t)+(\alpha-1)q(t)$ exists. According to Lemma $\ref{x-x}$, $\lim_{t\rightarrow +\infty} q(t)$ exists. It immediately follows that $\lim_{t\rightarrow +\infty}\|x(t)-x^*\|^2$ exists, since the limit of $\int_{t_0}^{t}\langle x(s)-x^*, \mu(s)\rangle ds$ exists by $(\ref{euql})$. And then $\lim_{t\rightarrow +\infty} t\langle \dot{x}(t)+\mu(t),x(t)-x^*\rangle$ exists as well.\\
Here we have completed the proof of $(iii)$.

To complete the proof via the Opial's Lemma, we need to prove that every weak sequential cluster point of $x(t)$ belongs to zer$\mathcal{A}$.
Let $t_n\rightarrow +\infty$ such that $x(t_n)\rightharpoonup \bar{x}$. From $(\ref{sorder})$, we have $\mathcal{A}_{\eta}x(t_n)\rightarrow 0$. Passing to the limit in 
$$\mathcal{A}_{\eta}x(t_n)\in \mathcal{A}(x(t_n)-\eta\mathcal{A}_{\eta}(x(t_n)))$$
and using Proposition $\ref{maximal}$, we obtain 
$$ 0\in \mathcal{A}(\bar{x}).$$
Consequently, $x(t)$ converges weakly to an element of zer$\mathcal{A}$, which completes the proof.
\end{proof}

\begin{remark}
	Theorem \ref{bounded} has shown that the dynamic system \eqref{DS} associated with a maxiamal comonotone operator has some similarities with the system $(\ref{sys6})$  considered in \cite{AttouchSC}, where  a maximal monotone operator is involved and the  Yosida regularization parameter $\lambda(t)$ be a  quadratic function of the time. As a comparision,   in our  dynamical system \eqref{DS}, the involved operator is maximal comonotone (not necessarily monotone), and the Yosida regularization parameter $\eta$ is a constant. 
\end{remark}

\section{A discrete algorithm}
In this section,by discretizing system $(\ref{DS})$,  we propose an algorithm for solving the problem \eqref{problem} with $\mathcal{A}$ being maximal comonotone.

Discretization of the system $(\ref{DS})$ with respect to the time variable $t$, with constant step size $h=1$, gives
$$x_{k+1}-x_k+\mathcal{A}_{\eta}x_{k+1}+(\frac{\alpha}{k}-1)(x_k-x_{k-1})+(\frac{\beta}{k}-1)\mathcal{A}_{\eta}x_k=0.$$
Equivalently,
\begin{equation}\label{al}
\mathcal{A}_{\eta}x_{k+1}=x_k-x_{k+1}+(1-\frac{\alpha}{k})(x_k-x_{k-1})+(1-\frac{\beta}{k})\mathcal{A}_{\eta}x_k.	
\end{equation}
Introducing intermediate variable $y_k:=x_{k+1}+\mathcal{A}_{\eta}x_{k+1}$  and according to Remark \ref{tuidao}, we have 
$$x_{k+1}=(1-\frac{1}{\eta+1})y_k+ \frac{1}{\eta+1}J_{\eta+1}^{\mathcal{A}}y_k.$$
Thus, $(\ref{al})$ can be equivalently rewritten as the following algorithm
\begin{equation}\label{ed}
\left\{
\begin{array}{rcl}
y_{k}&=&x_k+(1-\frac{\alpha}{k})(x_k-x_{k-1})+(1-\frac{\beta}{k})(y_{k-1}-x_k)\\
x_{k+1}&=&(1-\frac{1}{\eta+1})y_k+ \frac{1}{\eta+1}J_{\eta+1}^{\mathcal{A}}y_k,
\end{array}
\right.
\end{equation}
which combines relaxation factor $\frac{1}{\eta+1}$, a momentum term $(1-\frac{\alpha}{k})(x_k-x_{k-1})$ and a correction term $(1-\frac{\beta}{k})(y_{k-1}-x_k)$.

We will show that the algorithm \eqref{ed} owns convergence properties matching to that of the dynamical system \eqref{DS}. In order to simplify the proof, we will first derive some lemmas which will be used in the sequel.
\begin{lemma}\label{equa}
Let $\{x_k\}$ be a sequence generated by $(\ref{ed})$. Then the following equation is true
\begin{eqnarray*}
&-&s(k+1-\alpha)\langle x_{k+1}-x_k,x^*-x_{k+1}\rangle+s(k-\alpha)\langle  x^*-x_k,x_k-x_{k-1}\rangle\\
&=&s(k-\alpha)\langle x_{k+1}-x_k,x_k-x_{k-1}\rangle+sk\langle\mathcal{A}_{\eta}x_{k+1}-(1-\frac{\beta}{k})\mathcal{A}_{\eta}x_k,x^*-x_{k+1}\rangle \\
&&-\frac{s(\alpha-1)}{2}\|x_{k+1}-x_k\|^2-\frac{s(\alpha-1)}{2}\|x_{k+1}-x^*\|^2+\frac{s(\alpha-1)}{2}\|x^*-x_k\|^2.
\end{eqnarray*}
\end{lemma}

\begin{proof}
In view of $(\ref{al})$, it follows that
$$(1-\frac{\alpha}{k})(x_k-x_{k-1})=(x_{k+1}-x_k)+\mathcal{A}_{\eta}x_{k+1}-(1-\frac{\beta}{k})\mathcal{A}_{\eta}x_k.$$ 
Taking the scalar product of each sideof the above inequality with $k(x^*-x_{k+1})$, we derive that
\begin{eqnarray*}
&&(k-\alpha)\langle x^*-x_{k+1}, x_k-x_{k-1}\rangle\\
&=&k\langle x^*-x_{k+1}, x_{k+1}-x_k \rangle +k\langle \mathcal{A}_{\eta}x_{k+1}- (1-\frac{\beta}{k})\mathcal{A}_{\eta}x_k,x^*-x_{k+1}\rangle\\
&=&(k+1-\alpha)\langle x^*-x_{k+1}, x_{k+1}-x_k \rangle +k\langle \mathcal{A}_{\eta}x_{k+1}- (1-\frac{\beta}{k})\mathcal{A}_{\eta}x_k,x^*-x_{k+1}\rangle\\
&&+(\alpha-1)\langle x^*-x_{k+1},x_{k+1}-x_k \rangle.
\end{eqnarray*}
Multiply the above equation by $s$ to get the following formula
\begin{eqnarray*}
&&-s(k+1-\alpha)\langle x^*-x_{k+1}, x_{k+1}-x_k \rangle\\
&=&-s(k-\alpha)\langle x^*-x_{k+1}, x_k-x_{k-1}\rangle+sk\langle \mathcal{A}_{\eta}x_{k+1}- (1-\frac{\beta}{k})\mathcal{A}_{\eta}x_k,x^*-x_{k+1}\rangle\\
&&+s(\alpha-1)\langle x^*-x_{k+1},x_{k+1}-x_k \rangle.
\end{eqnarray*}
 This yields
\begin{eqnarray*}
&&-s(k+1-\alpha)\langle x_{k+1}-x_k,x^*-x_{k+1}\rangle+s(k-\alpha)\langle  x^*-x_k,x_k-x_{k-1}\rangle\\
&=&-s(k-\alpha)\langle x_k-x_{k-1},x^*-x_{k+1}\rangle+s(k-\alpha)\langle x^*-x_k,x_k-x_{k-1}\rangle\\
&&+sk\langle \mathcal{A}_{\eta}x_{k+1}-(1-\frac{\beta}{k})\mathcal{A}_{\eta}x_k,x^*-x_{k+1}\rangle+s(\alpha-1)\langle  x_{k+1}-x_{k},x^*-x_{k+1}\rangle\\
&=&s(k-\alpha)\langle x_{k+1}-x_k, x_k-x_{k-1} \rangle+sk\langle \mathcal{A}_{\eta}x_{k+1}-(1-\frac{\beta}{k})\mathcal{A}_{\eta}x_k,x^*-x_{k+1}\rangle\\
&&-s(\alpha-1)\langle x_{k+1}-x_k,x_{k+1}-x^*\rangle\\
&=&s(k-\alpha)\langle x_{k+1}-x_k,x_k-x_{k-1}\rangle+sk\langle\mathcal{A}_{\eta}x_{k+1}-(1-\frac{\beta}{k})\mathcal{A}_{\eta}x_k,x^*-x_{k+1}\rangle \\
&&-\frac{s(\alpha-1)}{2}\|x_{k+1}-x_k\|^2-\frac{s(\alpha-1)}{2}\|x_{k+1}-x^*\|^2+\frac{s(\alpha-1)}{2}\|x^*-x_k\|^2.
\end{eqnarray*}	
\end{proof}

\begin{lemma}\label{AA}
Let $\{x_k\}$ be a sequence generated by $(\ref{ed})$. Then the following equation is true
\begin{eqnarray*}
&-&(k^2-sk)\langle \mathcal{A}_{\eta}x_{k+1}-(1-\frac{\beta}{k})\mathcal{A}_{\eta}x_k,x_{k+1}-x_k\rangle-\frac{1}{2}k^2\|\mathcal{A}_{\eta}x_{k+1}-(1-\frac{\beta}{k})\mathcal{A}_{\eta}x_k\|^2\\
&=&s(k-\alpha)\langle x_{k+1}-x_k,x_k-x_{k-1}\rangle-(sk-\frac{1}{2}k^2)\|x_{k+1}-x_k\|^2\\
&&-\frac{1}{2}(k-\alpha)^2\|x_k-x_{k-1}\|^2.
\end{eqnarray*}
\end{lemma}

\begin{proof}
According to $(\ref{al})$ and expanding the following term, we get 
\begin{eqnarray*}
&&(1-\frac{s}{k})\langle\mathcal{A}_{\eta}x_{k+1}- (1-\frac{\beta}{k})\mathcal{A}_{\eta}x_k,x_{k+1}-x_k\rangle\\
&&+\frac{1}{2}\|x_{k+1}-x_k-(1-\frac{\alpha}{k})(x_k-x_{k-1})\|^2\\
&=&(1-\frac{s}{k})\langle (1-\frac{\alpha}{k})(x_k-x_{k-1})-(x_{k+1}-x_k),x_{k+1}-x_k\rangle\\
&&+\frac{1}{2}\|x_{k+1}-x_{k}-(1-\frac{\alpha}{k})(x_k-x_{k-1})\|^2\\
&=&(1-\frac{s}{k})(1-\frac{\alpha}{k})\langle x_k-x_{k-1}, x_{k+1}-x_k\rangle-(1-\frac{s}{k})\|x_{k+1}-x_k\|^2+\frac{1}{2}\|x_{k+1}-x_k\|^2\\
&&+\frac{1}{2}(1-\frac{\alpha}{k})^2\|x_k-x_{k-1}\|^2-(1-\frac{\alpha}{k})\langle x_{k+1}-x_k, x_k-x_{k-1}\rangle\\
&=&-\frac{s}{k}(1-\frac{\alpha}{k})\langle x_k-x_{k-1},x_{k+1}-x_k\rangle\\
&&+(\frac{s}{k}-\frac{1}{2})\|x_{k+1}-x_k\|^2+\frac{1}{2}(1-\frac{\alpha}{k})^2\|x_k-x_{k-1}\|^2.
\end{eqnarray*}
Multiplying this equality by $-k^2$, completes the proof.
\end{proof}

\begin{theorem}\label{kehe}
Under $(H)$, let $\{x_k\}$ be a sequence generated by $(\ref{ed})$ and $\alpha\geq\beta+1$ and $\beta>1$. Then the following properties hold:
\begin{itemize}
	\item [(i)] $\sum k^2\|\mathcal{A}_{\eta}x_{k+1}-\mathcal{A}_{\eta}x_k\|^2<+\infty$, $\sum k\|\mathcal{A}_{\eta}x_k\|^2<+\infty$, $\sum k^2\|\mathcal{A}_{\eta}x_{k+1}-(1-\frac{\beta}{k})\mathcal{A}_{\eta}x_k\|^2<+\infty$ and $\sum k\|x_{k+1}-x_k\|^2<+\infty$.
	\item [(ii)] $\lim_{k\rightarrow +\infty}k\|x_k-x_{k-1}\|=0$ and $\lim_{k\rightarrow +\infty}k\|\mathcal{A}_{\eta}x_k\|=0$.
	\item [(iii)] for any $x^*\in \mbox{zer}\mathcal{A}$, $\lim_{k\rightarrow +\infty}\|x_k-x^*\|^2$ exists.
	\item [(iv)] the sequence $\{x_k\}$ converges weakly, as $k\rightarrow +\infty$, to some $\hat{x}\in \mbox{zer}\mathcal{A}$.
\end{itemize}	
\end{theorem}

\begin{proof}
Take $x^*\in \mbox{zer}\mathcal{A}$. For $0 \leq s \leq \alpha-1 $, define the following function: 
\begin{eqnarray*}
\varepsilon_s(k)&:=&\frac{1}{2}\|s(x^*-x_k)-(k-\alpha)(x_k-x_{k-1})\|^2+\frac{s(\alpha-s-1)}{2}\|x_k-x^*\|^2\nonumber\\
&&+s(k-1)\langle \mathcal{A}_{\eta}x_k, x_k-x^* \rangle.
\end{eqnarray*}
The definiton of $\varepsilon_s(k)$ ensures that
\begin{eqnarray*}
&&\varepsilon_s(k+1)-\varepsilon_s(k)\\
&=&\frac{1}{2}\|s(x^*-x_{k+1})-(k+1-\alpha)(x_{k+1}-x_k)\|^2\\
&&-\frac{1}{2}\|s(x^*-x_k)-(k-\alpha)(x_k-x_{k-1})\|^2\\
&&+\frac{s(\alpha-s-1)}{2}\|x_{k+1}-x^*\|^2-\frac{s(\alpha-s-1)}{2}\|x_k-x^*\|^2\\
&&+sk\langle\mathcal{A}_{\eta}x_{k+1}, x_{k+1}-x^*\rangle-s(k-1)\langle \mathcal{A}_{\eta}x_k, x_k-x^* \rangle\\
&=&\frac{1}{2}s^2\|x^*-x_{k+1}\|^2+\frac{1}{2}(k+1-\alpha)^2\|x_{k+1}-x_k\|^2\\
&&-s(k+1-\alpha)\langle  x^*-x_{k+1},x_{k+1}-x_k\rangle\\
&&-\frac{1}{2}s^2\|x^*-x_k\|^2-\frac{1}{2}(k-\alpha)^2\|x_k-x_{k-1}\|^2+s(k-\alpha)\langle  x^*-x_k,x_k-x_{k-1}\rangle\\
&&+\frac{1}{2}s(\alpha-1)\|x_{k+1}-x^*\|^2-\frac{1}{2}s^2\|x_{k+1}-x^*\|^2-\frac{s(\alpha-1)}{2}\|x^*-x_k\|^2\\
&&+\frac{s^2}{2}\|x_k-x^*\|^2+sk\langle \mathcal{A}_{\eta}x_{k+1}, x_{k+1}-x^*\rangle-s(k-1)\langle \mathcal{A}_{\eta}x_k, x_k-x^* \rangle.\\
& \overset{\text{\mbox{Lemma}~\ref{equa}}}    =&\frac{1}{2}(k+1-\alpha)^2\|x_{k+1}-x_k\|^2-\frac{1}{2}(k-\alpha)^2\|x_k-x_{k-1}\|^2\\
&&+s(k-\alpha)\langle x_{k+1}-x_k, x_k-x_{k-1}\rangle\\
&&+sk\langle \mathcal{A}_{\eta}x_{k+1}- (1-\frac{\beta}{k})\mathcal{A}_{\eta}x_k,x^*-x_{k+1}\rangle-\frac{s(\alpha-1)}{2}\|x_{k+1}-x_k\|^2\\
&&+sk\langle \mathcal{A}_{\eta}x_{k+1}, x_{k+1}-x^*\rangle-s(k-1)\langle \mathcal{A}_{\eta}x_k, x_k-x^* \rangle.
\end{eqnarray*}
Then by Lemma $\ref{AA}$, we find
\begin{eqnarray*}
&&\varepsilon_{s}(k+1)-\varepsilon_s(k)\\
&=&(\frac{1}{2}k^2-sk)\|x_{k+1}-x_k\|^2-\frac{1}{2}(k-\alpha)^2\|x_k-x_{k-1}\|^2\\
&&+s(k-\alpha)\langle x_{k+1}-x_k,x_k-x_{k-1}\rangle+[\frac{1}{2}(k+1-\alpha)^2-\frac{1}{2}k^2+sk]\|x_{k+1}-x_k\|^2\\
&&+sk\langle \mathcal{A}_{\eta}x_{k+1}-(1-\frac{\beta}{k})\mathcal{A}_{\eta}x_k,x^*-x_{k+1}\rangle-\frac{1}{2}s(\alpha-1)\|x_{k+1}-x_k\|^2\\
&&+sk\langle  \mathcal{A}_{\eta}x_{k+1},x_{k+1}-x^*\rangle-s(k-1)\langle \mathcal{A}_{\eta}x_k,x_k-x^*\rangle\\
&=&-(k^2-sk)\langle  \mathcal{A}_{\eta}x_{k+1}-(1-\frac{\beta}{k})\mathcal{A}_{\eta}x_k,x_{k+1}-x_k\rangle\\
&&-\frac{1}{2}k^2\|\mathcal{A}_{\eta}x_{k+1}-(1-\frac{\beta}{k})\mathcal{A}_{\eta}x_k\|^2\\
&&+sk\langle \mathcal{A}_{\eta}x_{k+1}-(1-\frac{\beta}{k})\mathcal{A}_{\eta}x_k,x^*-x_{k+1}\rangle-\frac{1}{2}s(\alpha-1)\|x_{k+1}-x_k\|^2\\
&&+sk\langle \mathcal{A}_{\eta}x_{k+1}, x_{k+1}-x^*\rangle- s(k-1)\langle \mathcal{A}_{\eta}x_k, x_k-x^*\rangle\\
&&+[(1+s  -\alpha)k+\frac{1}{2}(1-\alpha)^2]\|x_{k+1}-x_k\|^2.
\end{eqnarray*}
Obsever that
\begin{eqnarray*}
&&-(k-s)k\langle \mathcal{A}_{\eta}x_{k+1}-(1-\frac{\beta}{k})\mathcal{A}_{\eta}x_k ,x_{k+1}-x_k\rangle\\
&&-sk\langle \mathcal{A}_{\eta}x_{k+1}- (1-\frac{\beta}{k})\mathcal{A}_{\eta}x_k,x_{k+1}-x^*\rangle\\
&=&-(k-s)k\langle \mathcal{A}_{\eta}x_{k+1}-\mathcal{A}_{\eta}x_k,x_{k+1}-x_k \rangle-(k-s)k\langle \frac{\beta}{k}\mathcal{A}_{\eta}x_k,x_{k+1}-x_k\rangle\\
&&-sk\langle \mathcal{A}_{\eta}x_{k+1},x_{k+1}-x^*\rangle+s(k-\beta)\langle \mathcal{A}_{\eta}x_k,x_{k+1}-x_k+x_k-x^*\rangle	\\
&=&-(k-s)k\langle \mathcal{A}_{\eta}x_{k+1}-\mathcal{A}_{\eta}x_k,x_{k+1}-x_k\rangle-(k-s)\beta\langle \mathcal{A}_{\eta}x_k,x_{k+1}-x_k\rangle\\
&&-sk\langle \mathcal{A}_{\eta}x_{k+1},x_{k+1}-x^*\rangle +s(k-\beta)\langle \mathcal{A}_{\eta}x_k,x_{k+1}-x_k\rangle\\
&&+s(k-\beta)\langle \mathcal{A}_{\eta}x_k,x_k-x^*\rangle\\
&=&-(k-s)k\langle \mathcal{A}_{\eta}x_{k+1}-\mathcal{A}_{\eta}x_k,x_{k+1}-x_k\rangle+(s-\beta)k\langle \mathcal{A}_{\eta}x_k, x_{k+1}-x_k\rangle\\
&&-sk\langle\mathcal{A}_{\eta}x_{k+1},x_{k+1}-x^*\rangle+s(k-\beta)\langle\mathcal{A}_{\eta}x_k,x_k-x^*\rangle.
\end{eqnarray*}
Combining these two results, we establish that
\begin{eqnarray}\label{e}
&&\varepsilon_{s}(k+1)-\varepsilon_s(k)\nonumber\\
&=&-(k-s)k\langle\mathcal{A}_{\eta}x_{k+1}-\mathcal{A}_{\eta}x_k, x_{k+1}-x_k \rangle-sk\langle\mathcal{A}_{\eta}x_{k+1},x_{k+1}-x^* \rangle\nonumber\\
&&+s(k-\beta)\langle\mathcal{A}_{\eta}x_k,x_k-x^* \rangle+(s-\beta)k\langle  \mathcal{A}_{\eta}x_k,x_{k+1}-x_k\rangle\nonumber\\
&&+sk\langle\mathcal{A}_{\eta}x_{k+1},x_{k+1}-x^*\rangle-s(k-1)\langle\mathcal{A}_{\eta}x_k,x_k-x^*\rangle\nonumber\\
&&-\frac{1}{2}k^2\|\mathcal{A}_{\eta}x_{k+1}-(1-\frac{\beta}{k})\mathcal{A}_{\eta}x_k\|^2-\frac{1}{2}s(\alpha-1)\|x_{k+1}-x_k\|^2\nonumber\\
&&+[(1+s-\alpha)k+\frac{1}{2}(1-\alpha)^2]\|x_{k+1}-x_k\|^2\nonumber\\
&=&-(k-s)k\langle\mathcal{A}_{\eta}x_{k+1}-\mathcal{A}_{\eta}x_k, x_{k+1}-x_k \rangle-\frac{1}{2}k^2\|\mathcal{A}_{\eta}x_{k+1}-(1-\frac{\beta}{k})\mathcal{A}_{\eta}x_k\|^2\nonumber\\
&&+s(1-\beta)\langle \mathcal{A}_{\eta}x_k,x_k-x^*\rangle+(s-\beta)k\langle \mathcal{A}_{\eta}x_k,x_{k+1}-x_k\rangle\nonumber\\
&&+[(1+s-\alpha)k+\frac{1}{2}\alpha(\alpha-1)]\|x_{k+1}-x_k\|^2.
\end{eqnarray}

Take $s=\beta$.  According to Proposition $\ref{z}$-$(iv)$, the above equality yields
\begin{eqnarray*}
&&\varepsilon_{\beta}(k+1)-\varepsilon_{\beta}(k)\\
&=&-(k-\beta)k\langle \mathcal{A}_{\eta}x_{k+1}-\mathcal{A}_{\eta}x_k, x_{k+1}-x_k\rangle+\beta(1-\beta)\langle \mathcal{A}_{\eta}x_k,x_k-x^*\rangle\\
&&-\frac{1}{2}k^2\|\mathcal{A}_{\eta}x_{k+1}-(1-\frac{\beta}{k})\mathcal{A}_{\eta}x_k\|^2+[(1+\beta-\alpha)k+\frac{1}{2}\alpha(\alpha-1)]\|x_{k+1}-x_k\|^2\\
&\leq& -(\rho+\eta)(k-\beta)k\| \mathcal{A}_{\eta}x_{k+1}-\mathcal{A}_{\eta}x_k\|^2+\beta(1-\beta)\| \mathcal{A}_{\eta}x_k\|^2\\
&&-\frac{1}{2}k^2\|\mathcal{A}_{\eta}x_{k+1}-(1-\frac{\beta}{k})\mathcal{A}_{\eta}x_k\|^2+[(1+\beta-\alpha)k+\frac{1}{2}\alpha(\alpha-1)]\|x_{k+1}-x_k\|^2.
\end{eqnarray*}
Together with $\beta >1$ and $0\leq s\leq \alpha-1$, it follows immediately that the nonnegative sequence $\varepsilon_{\beta}(k)$ is nonincreasing. Hence, $\varepsilon_{\beta}(k)$ is convergent and bounded. Furthermore, by adding the above inequality from $k=1$ to $k=+\infty$, we obtain 
\begin{equation}\label{1}
\sum (k-\beta)k\| \mathcal{A}_{\eta}x_{k+1}-\mathcal{A}_{\eta}x_k\|^2<+\infty,
\end{equation}
\begin{equation}\label{2}
\sum \|\mathcal{A}_{\eta}x_k\|^2<+\infty,	
\end{equation}
\begin{equation}\label{3}
\sum k^2\|\mathcal{A}_{\eta}x_{k+1}-(1-\frac{\beta}{k})\mathcal{A}_{\eta}x_k\|^2<+\infty,	
\end{equation}
\begin{equation}\label{4}
\sum k\|x_{k+1}-x_k\|^2<+\infty.	
\end{equation}

Take $s\neq\beta$. Using $(\ref{al})$, it ensues that
\begin{eqnarray*}
&&\|x_{k+1}-x_k-(x_k-x_{k-1})\|^2\\
&=&\|x_{k+1}-x_k-(1-\frac{\alpha}{k})(x_k-x_{k-1})-\frac{\alpha}{k}(x_k-x_{k-1})\|^2\\
&=&\|\mathcal{A}_{\eta}x_{k+1}-(1-\frac{\beta}{k})\mathcal{A}_{\eta}x_k-\frac{\alpha}{k}(x_k-x_{k-1})\|^2\\
&\leq&2\|\mathcal{A}_{\eta}x_{k+1}-(1-\frac{\beta}{k})\mathcal{A}_{\eta}x_k\|^2+2\frac{{\alpha}^2}{k^2}\|(x_k-x_{k-1})\|^2.
\end{eqnarray*}
Multiplying both sides of by $k^2$, we obtain
$$k^2\|x_{k+1}-x_k-(x_k-x_{k-1})\|^2\leq 2k^2\|\mathcal{A}_{\eta}x_{k+1}-(1-\frac{\beta}{k})\mathcal{A}_{\eta}x_k\|^2+2{\alpha}^2\|x_k-x_{k-1}\|^2.$$
The inequality above, combined with $(\ref{3})$ and $(\ref{4})$ yields
$$\sum k^2\|x_{k+1}-x_k-(x_k-x_{k-1})\|^2<+\infty.$$
Taking $(\ref{al})$ into account, we establish that
\begin{eqnarray}\label{kehele}
&&k^2\|x_{k+1}-x_k+\mathcal{A}_{\eta}x_{k+1}\|^2-(k-1)^2\|x_k-x_{k-1}+\mathcal{A}_{\eta}x_k\|^2\nonumber\\
&=&k^2\|x_k-x_{k-1}+\mathcal{A}_{\eta}x_k-\frac{\alpha}{k}(x_k-x_{k-1})-\frac{\beta}{k}\mathcal{A}_{\eta}x_k\|^2\nonumber\\
&&-(k-1)^2\|x_k-x_{k-1}+\mathcal{A}_{\eta}x_k\|^2\nonumber\\
&=&k^2\|(1-\frac{\beta}{k})(x_k-x_{k-1}+\mathcal{A}_{\eta}x_k)+\frac{\beta-\alpha}{k}(x_k-x_{k-1})\|^2\nonumber\\
&&-(k-1)^2\|x_k-x_{k-1}+\mathcal{A}_{\eta}x_k\|^2\nonumber\\
&=&[(k-\beta)^2-(k-1)^2]\|x_k-x_{k-1}+\mathcal{A}_{\eta}x_k\|^2+(\beta-\alpha)^2\|x_k-x_{k-1}\|^2\nonumber\\
&&+2(k-\beta)(\beta-\alpha)\langle x_k-x_{k-1}+\mathcal{A}_{\eta}x_k,x_k-x_{k-1}\rangle\nonumber\\
&\leq&[(2-2\beta)k+{\beta}^2-1]\|x_k-x_{k-1}+\mathcal{A}_{\eta}x_k\|^2+(\beta-\alpha)^2\|x_k-x_{k-1}\|^2\nonumber\\
&&2(k-\beta)(\alpha-\beta)|\langle x_k-x_{k-1}+\mathcal{A}_{\eta}x_k,x_k-x_{k-1}\rangle|\nonumber\\
&\leq&[(2-2\beta)k+{\beta}^2-1]\|x_k-x_{k-1}+\mathcal{A}_{\eta}x_k\|^2+(\beta-\alpha)^2\|x_k-x_{k-1}\|^2\nonumber\\
&&+(k-\beta)(\alpha-\beta)\xi\|x_k-x_{k-1}+\mathcal{A}_{\eta}x_k\|^2+(k-\beta)(\alpha-\beta)\frac{1}{\xi}\|x_k-x_{k-1}\|^2\nonumber\\
&=&[(2-2\beta)k+{\beta}^2-1+(k-\beta)(\alpha-\beta)\xi]\|x_k-x_{k-1}+\mathcal{A}_{\eta}x_k\|^2\nonumber\\
&&+[(\beta-\alpha)^2+(k-\beta)|\alpha-\beta|\frac{1}{\xi}]\|x_k-x_{k-1}\|^2\nonumber\\
&\leq&[(2(1-\beta)+(\alpha-\beta)\xi)k+{\beta}^2-1]\|x_k-x_{k-1}+\mathcal{A}_{\eta}x_k\|^2\nonumber\\
&&+[(\beta-\alpha)^2+(k-\beta)(\alpha-\beta)\frac{1}{\xi}]\|x_k-x_{k-1}\|^2.
\end{eqnarray}
By assumption $\alpha\geq\beta+1$ and $\beta>1$, let us choose $\xi\in (0,\frac{2(\beta-1)}{\alpha-\beta})$, such that $2(1-\beta)+(\alpha-\beta)\xi<0$. By neglecting the nonpositive term $[(2(1-\beta)+(\alpha-\beta)\xi)k+{\beta}^2-1]\|x_k-x_{k-1}+\mathcal{A}_{\eta}x_k\|^2$, $(\ref{kehele})$ becomes $k^2\|x_{k+1}-x_k+\mathcal{A}_{\eta}x_{k+1}\|^2-(k-1)^2\|x_k-x_{k-1}+\mathcal{A}_{\eta}x_k\|^2\leq [(\beta-\alpha)^2+(k-\beta)(\alpha-\beta)\frac{1}{\xi}]\|x_k-x_{k-1}\|^2$. Now applying  Lemma $\ref{des}$ to this inequality, we establish that $\lim_{k\rightarrow +\infty}k^2\|x_k-x_{k-1}+\mathcal{A}_{\eta}x_k\|^2$ exists, because the right-hand side is summable by $(\ref{4})$. Furthermore, by adding the above inequality from $k=1$ to $k=+\infty$, we obtain  
\begin{equation}\label{5}
\sum k\|x_k-x_{k-1}+\mathcal{A}_{\eta}x_k\|^2<+\infty.	
\end{equation}
Indeed, the limit of $k^2\|x_k-x_{k-1}+\mathcal{A}_{\eta}x_k\|^2$ must be zero. Otherwise, if its limit is a non-zero constant, that would contradict $(\ref{5})$.\\
This guarantees  that
\begin{equation}\label{0}
\lim_{k\rightarrow +\infty}k^2\|x_k-x_{k-1}+\mathcal{A}_{\eta}x_k\|^2=0.	
\end{equation}
Note that
\begin{eqnarray*}
&&k\|\mathcal{A}_{\eta}x_k\|^2\\
&=&k\|x_k-x_{k-1}+\mathcal{A}_{\eta}x_k-(x_k-x_{k-1})\|^2\\
&\leq&2k\||x_k-x_{k-1}+\mathcal{A}_{\eta}x_k\|^2+2k\|x_k-x_{k-1}\|^2.
\end{eqnarray*}
Now use $(\ref{4})$ and $(\ref{5})$ to obtain
$$\sum k\|\mathcal{A}_{\eta}x_k\|^2<+\infty.$$
Combining the above result and $(\ref{4})$ with $k|\langle \mathcal{A}_{\eta}x_k, x_{k+1}-x_k\rangle|\leq \frac{k}{2}\| \mathcal{A}_{\eta}x_k\|^2+\frac{k}{2}\|x_{k+1}-x_k\|^2$, we find
\begin{eqnarray}\label{khe}
\sum k|\langle \mathcal{A}_{\eta}x_k,x_{k+1}-x_k\rangle|<+\infty.	
\end{eqnarray}
Here we have completed the proof of $(i)$.

Neglecting nonpositive terms of the right-hand side of $(\ref{e})$, we derive
\begin{eqnarray*}
&&\varepsilon_{s}(k+1)-\varepsilon_s(k)\\
&\leq &(s-\beta)k\langle \mathcal{A}_{\eta}x_k,x_{k+1}-x_k\rangle\\
&\leq& |s-\beta|k|\langle \mathcal{A}_{\eta}x_k,x_{k+1}-x_k\rangle|.	
\end{eqnarray*} 
$$$$
Connecting with $(\ref{khe})$ and applying Lemma $\ref{des}$ with $a_k=\varepsilon_s(k)$ and $b_k=k|\langle \mathcal{A}_{\eta}x_k,x_{k+1}-x_k\rangle|$, we know $\lim_{k\rightarrow +\infty}\varepsilon_s(k)$ exists.

Collecting these results, we deduce that $\forall 0\leq s\leq \alpha-1 $, $\lim_{k\rightarrow +\infty}\varepsilon_s(k)$ exists.
Take $s=0$, $\lim_{k\rightarrow +\infty}\varepsilon_k(0)=\lim_{k\rightarrow +\infty}\frac{1}{2}(k-\alpha)^2\|x_k-x_{k-1}\|^2$ exists.
Indeed, the limit of $(k-\alpha)^2\|x_k-x_{k-1}\|^2$ must be zero. Otherwise, if its limit is a non-zero constant, it would contradict $(\ref{4})$.
This guarantees that
\begin{equation}\label{eq0}
\lim_{k\rightarrow +\infty}\frac{1}{2}(k-\alpha)^2\|x_k-x_{k-1}\|^2=0.	
\end{equation}
Furthermore, we can improve the estimate of $\|\mathcal{A}_{\eta}x_k\|^2$.
Note that
\begin{eqnarray*}
&&(k-\alpha)^2\| \mathcal{A}_{\eta}x_k\|^2\\
&=&(k-\alpha)^2\| \mathcal{A}_{\eta}x_k+x_k-x_{k-1}-(x_k-x_{k-1})\|^2\\
&\leq&2(k-\alpha)^2\| \mathcal{A}_{\eta}x_k+x_k-x_{k-1}\|^2+2(k-\alpha)^2\|x_k-x_{k-1}\|^2.
\end{eqnarray*}
This together with $(\ref{0})$ and $(\ref{eq0})$ yields
$$\lim_{k\rightarrow +\infty}(k-\alpha)^2\|\mathcal{A}_{\eta}x_k\|^2=0,$$
which completes the proof of $(ii)$.

Combining this result with the boundedness of $\|x_k-x^*\|$, we have
$$\lim_{k\rightarrow +\infty}(k-1)\langle \mathcal{A}_{\eta}x_k,x_k-x^*\rangle=0,$$
because of $(k-1)\langle \mathcal{A}_{\eta}x_k,x_k-x^*\rangle \leq (k-1)\|\mathcal{A}_{\eta}x_k\|\|x_k-x^*\|.$
By virtue of the definition of $\varepsilon_{\beta}(k)$, we get
\begin{eqnarray*}
&&\varepsilon_{\beta}(k)\\
&=&\frac{1}{2}\|\beta(x^*-x_k)-(k-\alpha)(x_k-x_{k-1})\|^2\\
&&+\beta(k-1)\langle \mathcal{A}_{\eta}x_k,x_k-x^*\rangle+\frac{\beta(\alpha-\beta-1)}{2}\|x_k-x^*\|^2\\
&=&\frac{\beta(\alpha-1)}{2}\|x^*-x_k\|^2-\beta(k-\alpha)\langle x^*-x_k,x_k-x_{k-1}\rangle \\
&&+\frac{1}{2}(k-\alpha)^2\|x_k-x_{k-1}\|^2+\beta(k-1)\langle \mathcal{A}_{\eta}x_k,x_k-x^*\rangle.
\end{eqnarray*}
In order to prove the existence of $\|x_k-x^*\|^2$, we just have to prove that the limit of $(k-\alpha)\langle x^*-x_k, x_k-x_{k-1}\rangle$ is zero.
According to Cauchy-Schwarz inequality, it is easy to see that
\begin{eqnarray*}
&&-\beta(k-\alpha)\|x^*-x_k\|\|x_k-x_{k-1}\|\leq-\beta(k-\alpha)|\langle x^*-x_k, x_k-x_{k-1}\rangle| \\
&\leq&-\beta(k-\alpha)\langle x^*-x_k, x_k-x_{k-1}\rangle\leq \beta(k-\alpha)\|x^*-x_k\|\|x_k-x_{k-1}\|.	
\end{eqnarray*}
Combining the boundedness of $\|x^*-x_k\|$ and $(\ref{eq0})$, we get
$$\lim_{k\rightarrow +\infty}-\beta(k-\alpha)\langle x^*-x_k, x_k-x_{k-1}\rangle=0.$$
Hence,$\lim_{k\rightarrow +\infty}\varepsilon_{\beta}(k)=\lim_{k\rightarrow +\infty}\frac{1}{2}(\alpha-1)\beta\|x_k-x^*\|^2$ exists, which completes the proof of $(iii)$
Let $\hat{x}$ be a weak cluster point of $\{x_k\}$, namely there exists a subsequence $\{x_{k_n}\}$ that weakly converges to $\hat{x}$. We have $\mathcal{A}_{\eta}x_{k_n}\rightarrow 0$. Passing to the limit to
$$\mathcal{A}_{\eta} x_{k_n}\in \mathcal{A}(x_{k_n}-\eta\mathcal{A}_{\eta} x_{k_n})$$ 
and using Proposition 4, we obtain 
$$0\in \mathcal{A}(\hat{x}),$$
which completes the proof.
\end{proof}

\begin{remark}
 Our algorithm \eqref{ed} is similar to  CRIPA-S algorithm $(\ref{CRIPAS})$, which is proposed by  Maing$\acute{e}$ \cite{Mainge} for solving a maximal  monotone inclusion problem. Theorem \ref{kehe} has shown that our algorithm \eqref{ed} enjoys convergence properties similar to  CRIPA-S algorithm $(\ref{CRIPAS})$. It is worth mentioning that although similar to  CRIPA-S algorithm $(\ref{CRIPAS})$, our algorithm  \eqref{ed} is obtained by discretization of the dynamical system $(\ref{DS})$ and can solve the maximal comonotone inclusion problem. 
\end{remark}

\section{Numerical experiments}

In this section, we illustrate the validity of the proposed dynamical system $(\ref{DS})$ as well as the resulting  algorithm $(\ref{ed})$ for solving the inclusion problem $(\ref{problem})$ by two examples. The simulations are conducted in Matlab (version 9.4.0.813654)R2018a. All the numerical procedures are performed on a personal computer with Inter(R) Core(TM) i7-4600U, CORES 2.69GHz and RAM 8.00GB.

The first example (Example \ref{Ex1}) is taken from \cite{AttouchSC},  by which  Attouch et. al. \cite{AttouchSC} test the dynamical system $(\ref{sys6})$ for the maximal monotone inclusion problem. Next, we will test our dynamical system \eqref{DS} and the resulting algorithm \eqref{ed}  in Example \ref{Ex1}, and make comparisons with  $(\ref{sys6})$ and  CRIPA-S algorithm $(\ref{CRIPAS})$, respectively.

\begin{example}\label{Ex1}
Let $\mathcal{A}: \mathbb{R}^2\rightarrow \mathbb{R}^2$ and $\mathcal{A}(x,y)=(-y,x)$, which is a linear skew symmetric operator. Clearly, $\mathcal{A}$ is a maximally monotone whose single zero is $x^*=(0,0)$. Further, $\mathcal{A}$ and its Yosida regularization $\mathcal{A}_{\eta}$ can be identified respectively with the matrices
$\mathcal{A}=
\begin{pmatrix}
0 & -1\\
1 & 0
\end{pmatrix}$,
 $\mathcal{A}_{\eta}=
\begin{pmatrix}
\frac{\eta}{{\eta}^2+1} & -\frac{1}{{\eta}^2+1}\\
\frac{1}{{\eta}^2+1} & \frac{\eta}{{\eta}^2+1}
\end{pmatrix}$.
\end{example}

Take $b=1$, $\lambda(t)=\lambda t^2$ and $e(t)\equiv 0$ in the system $(\ref{sys6})$.  Figure \ref{figure1} depicts the asymptotical behavior of the  system $(\ref{sys6})$  for different $\lambda$  and the system $(\ref{DS})$ for different $\eta$, respectively, with the same parameters $\alpha=2.5$, $t_0=0.1$ and $x(t_0)=\dot{x}(t_0)=(1,1)$, where $\beta=1.5$ in the system $(\ref{DS})$. As shown in Figure \ref{figure1}, the dynamical system \eqref{DS}outperform the dynamical system \eqref{sys6}.

Next,  by this example, we compare our algorithm  \eqref{ed}  with  the CRIPA-S algorithm $(\ref{CRIPAS})$ proposed by  Maing$\acute{e}$ \cite{Mainge}.   Figure \ref{figure2}  displays  the profiles of $\|x_{k}-x^*\|$ for the sequences $\{x_k\}$ generated by the algorithm \eqref{ed} for different   $\eta$ and the CRIPA-S algorithm  \eqref{CRIPAS}) for different $\lambda$, respectively, under the same stopping criteria $\|x_k-x^*\|\leq 10^{-7}$. The profile obtained for CRIPA-S algorithm $(\ref{CRIPAS})$ with the parameters $b=1$, $a_1=5.25$, $a_2=2.5$, $\bar{c}=7.875$ and starting points $x_0=x_{-1}=z_{-1}=(1,-1)$. In order to ensure $\lambda(1+k_0)=\eta+1$, we take $k_0=\frac{\eta+1}{\lambda}-1$. In the algorithm  \eqref{ed}, we take $\alpha=a_1$, $\beta=a_2$ and starting points $x_1=x_0=(1,-1)$. Figure \ref{figure2} shows that our algorithm $(\ref{ed})$ outperforms the CRIPA-S algorithm $(\ref{CRIPAS})$ in \cite{Mainge}.

In the following example, we test the dynamical system $(\ref{DS})$ and  the  algorithm $(\ref{ed})$ for solving the  maximal comonotone inclusion problem. 

\begin{example}\label{Ex2} Consider the following inclusion problem
$$0\in \mathcal{A}(x^*),$$
 where $\mathcal{A}=
\begin{pmatrix}
-\frac{2}{5} & \frac{4}{5}\\
-\frac{4}{5} & -\frac{2}{5}
\end{pmatrix}$
is a  maximal $\rho$-comonotone operator with $\rho= -\frac{1}{2}$. It is easily verified that $\mathcal{A}$ has a single zero $x^*=(0,0)$.	
\end{example}
The dynamical system $(\ref{sys6})$ proposed by  Attouch et. al. \cite{AttouchSC}  and CRIPA-S algorithm $(\ref{CRIPAS})$  proposed by  Maing$\acute{e}$ \cite{Mainge} are not applicable for this example because the operator $ \mathcal{A}$ is not monotone. We first use the dynamical system $(\ref{DS})$ to solve Example \ref{Ex2}.   Figure \ref{figure3} depicts the asymptotical behavior of the trajectort generated by $\eqref{DS}$ where it is solved with the ode45 function in Matlab on the interval $[0.1,100]$, with parameters $\alpha=3$, $\beta=2$, $\eta=2$ and $x(t_0)=\dot{x}(t_0)=(1,1)$. Figure \ref{figure4}  depicts the asymptotical behavior of the trajectory $x(t)$ generated by  $(\ref{DS})$ with different $\beta$ where it is solved  with the ode45 function in Matlab on the interval $[0.1,50]$ with parameters $\alpha=20$, $\eta=2$ and $x(t_0)=\dot{x}(t_0)=(1,1)$.

Now we use the algorithm $(\ref{al})$ to solve Example \ref{Ex2}.   Figure \ref{figure5} displays  the profiles of $\|x_{k}-x^*\|$ for the sequences $\{x_k\}$ generated by the  algorithm $(\ref{al})$ with different $\beta$ under the same stopping criteria $\|x_k-x^*\|\leq 10^{-7}$, $\alpha=10$, $\eta=2$ and $x_1=x_0=(1,1)$.

\begin{figure}[htp]
\begin{center}
\includegraphics[width=0.6\textwidth]{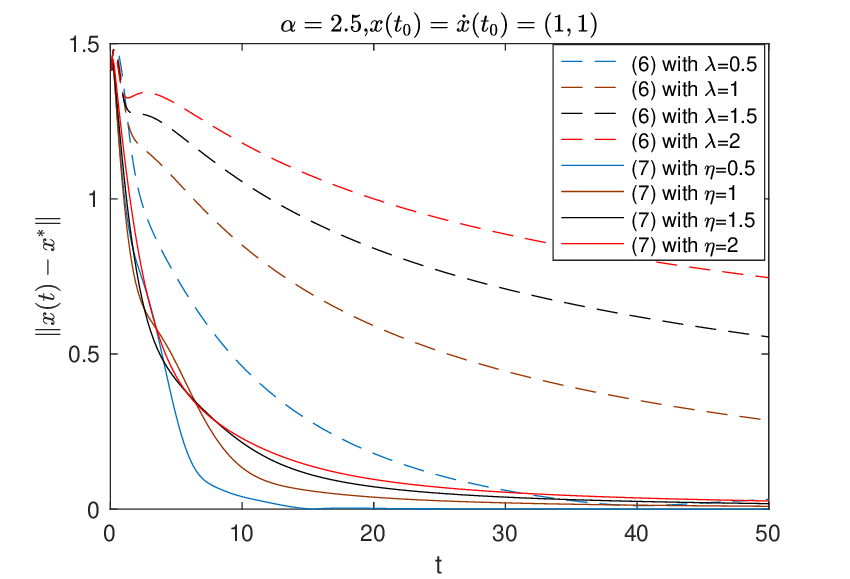}
\caption{Comparison of iteration error $\|x(t)-x^*\|$ for systems $(\ref{sys6})$ and $(\ref{DS})$  in Example $\ref{Ex1}$.}
\label{figure1}
\end{center}
\end{figure}

\begin{figure}[htp]
\begin{center}
\includegraphics[width=0.6\textwidth]{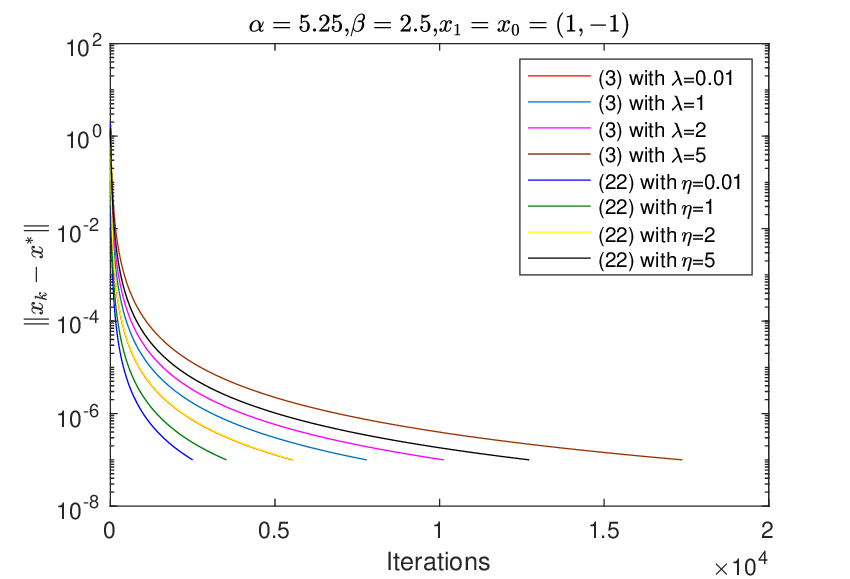}
\caption{Comparison of the iteration error $\|x_k-x^*\|$ for algorithms $(\ref{CRIPAS})$ and $(\ref{ed})$ in Example $\ref{Ex1}$.}
\label{figure2}
\end{center}
\end{figure}

\begin{figure}[htp]
\begin{center}
\includegraphics[width=0.6\textwidth]{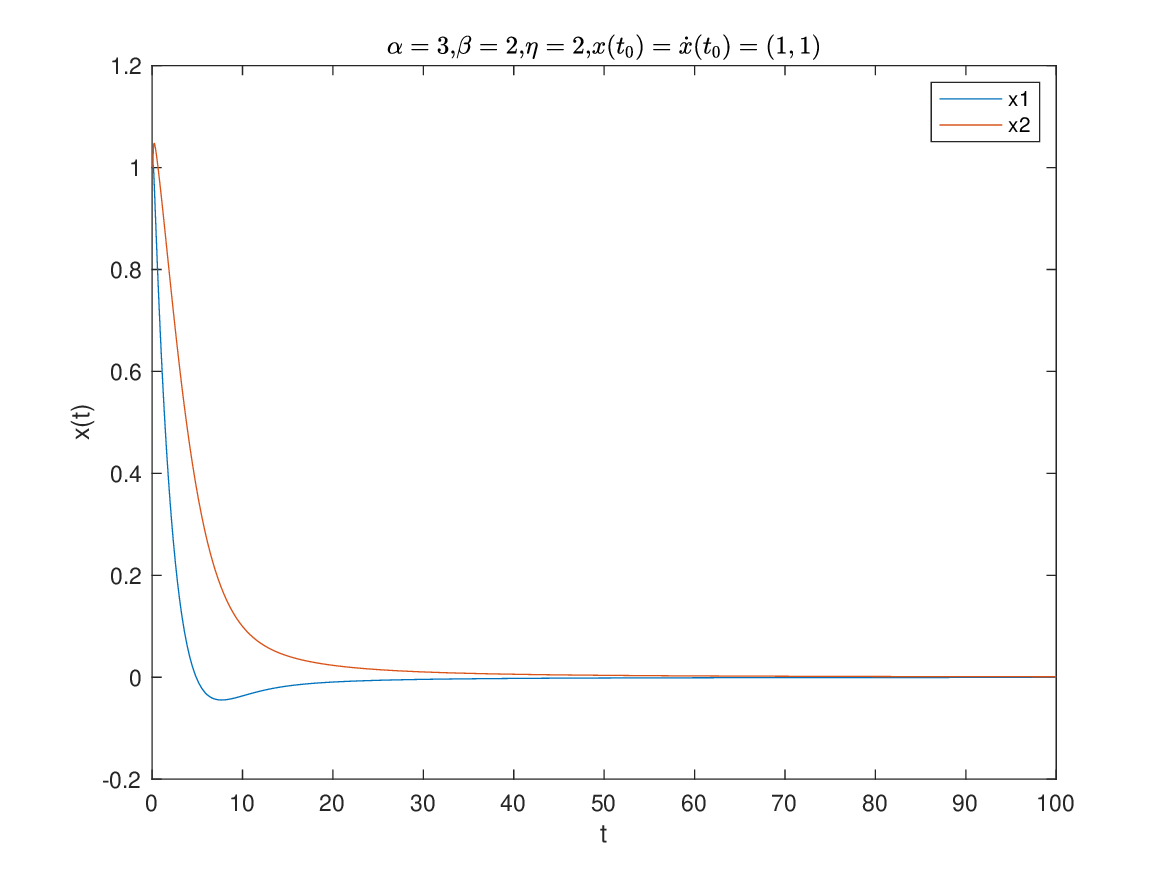}
\caption{Transient behavior of dynamical system $(\ref{DS})$  in Example $\ref{Ex2}$.}
\label{figure3}
\end{center}
\end{figure}

\begin{figure}[htp]
\begin{center}
\includegraphics[width=0.6\textwidth]{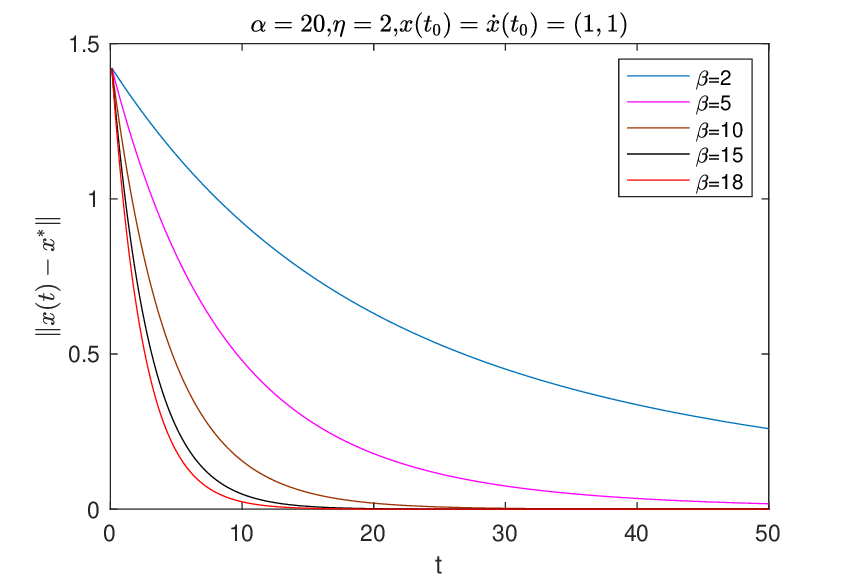}
\caption{Iteration error $\|x(t)-x^*\|$ for dynamical systems $(\ref{DS})$  with  different $\beta$  in Example $\ref{Ex2}$.}
\label{figure4}
\end{center}
\end{figure}

\begin{figure}[htp]
\begin{center}
\includegraphics[width=0.6\textwidth]{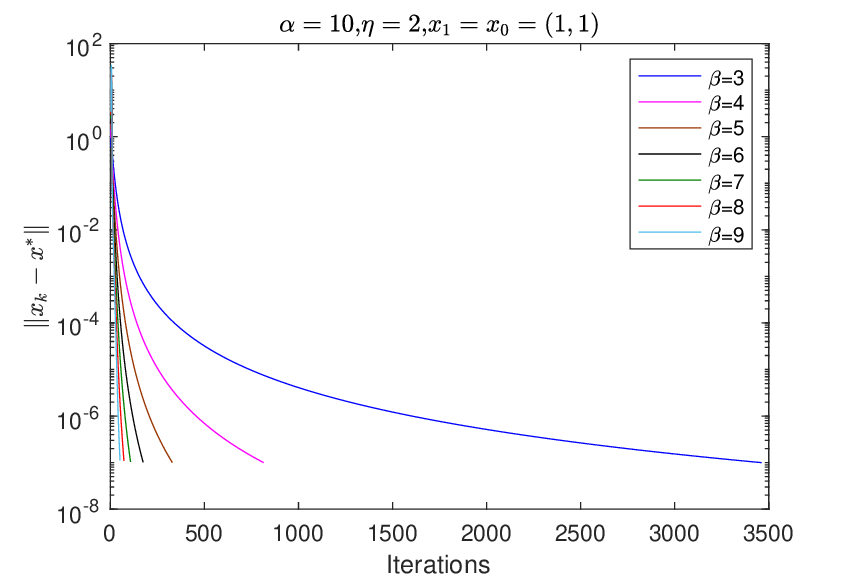}
\caption{Iteration error $\|x_k-x^*\|$ for algorithm $(\ref{ed})$  with  different $\beta$ in Example $\ref{Ex2}$.}
\label{figure5}
\end{center}
\end{figure}


\section{Concluding remarks}

In this paper we propose a second order dynamical system for solving a maximal comonotone inclusion problem and  prove convergence properties of the trajectories generated by the dynamical system under mild conditions. These convergence properties are similar to the ones of a second order dynmaical system considered by  Attouch et. al. \cite{AttouchSC}  for a maximal monotone inclusion problem.  By disretizing our dynmaical system, we propose an algorithm  which combines the relaxation factor, a momentum term and a correction term for solvig the maximal comonotone inclusion problem. Our algorithm is  similar to CRIPA-S algorithm $(\ref{CRIPAS})$ proposed by  Maing$\acute{e}$ \cite{Mainge} for the maximal monotone inclusion problem. By  numerical experiment, we demonstrate the validity of the proposed dynamical system and its resulting algorithm.

\section*{Disclosure statement}

No potential conflict of interest was reported by the authors.






\end{document}